\documentclass{amsart}
\usepackage{pgf,pgfarrows,pgfnodes,pgfautomata,pgfheaps,pgfshade,hyperref, amssymb,enumerate,amsmath}
\usepackage[all]{xy}
\usepackage[capitalize]{cleveref}
\usepackage{mathtools}
\usepackage[shortlabels]{enumitem}
\usepackage{tcolorbox}
\usepackage{bm}
\usepackage{csquotes}

\SetLabelAlign{parright}{\smash{\parbox[t]{\labelwidth}{\raggedleft#1}}}

\newtheorem{theorem}{Theorem}[section]
\newtheorem*{theorem*}{theorem}
\newtheorem{proposition}[theorem]{Proposition}
\newtheorem{lemma}[theorem]{Lemma}
\newtheorem{corollary}[theorem]{Corollary}

\newtheorem{remark}[theorem]{Remark}

\usepackage{euscript,epsfig}
\usepackage{tikz}
\usetikzlibrary{graphs}
\usetikzlibrary[graphs]
\usetikzlibrary{arrows}
\usepackage{tikz-cd}
\usetikzlibrary{decorations.markings}
\usepackage[colorinlistoftodos]{todonotes}

\def\R{{\mathbb R}}
\def\Z{{\mathbb Z}}

\def\C{{\mathbb C}}
\def\N{{\mathbb N}}

\def\S{{\mathbb S}}

\def\Z{{\mathbb Z}}

\def\T{{\mathbb T}}
\def\pP{{\mathbb P}}

\def\cB{{\mathcal B}}

\def \eps {\varepsilon}

\begin{document}

\title[Ratio limits and pressure function]
{Ratio limits and pressure function for group extensions of Gibbs Markov maps}
\date{\today}

\author{Jaime Gomez}
\address{Mathematisch Instituut,
University of Leiden, Einsteinweg 55,
2333 CC,  Leiden, Netherlands.}
\email{j.a.gomez.ortiz@math.leidenuniv.nl}

\author{Dalia Terhesiu}
\address{Mathematisch Instituut,
University of Leiden, Einsteinweg 55,
2333 CC,  Leiden, Netherlands.}
\email{daliaterhesiu@gmail.com}

\begin{abstract}
   Ratio limit theorems for random walks on (various) groups are known.  We obtain  a generalization of this type of ratio limit for deterministic walks on certain groups driven by Gibbs Markov  maps. In terms of proofs, the main difficulty comes down to the absence of a convolution structure. 
   Also, for (finitely generated) group extensions of Gibbs Markov maps we 
   obtain a characterization of the pressure function without a symmetry assumption.
   
\end{abstract}
\maketitle

\section{Introduction}
 We are interested in an analogue  of ratio local limit theorems in the sense of Stone~\cite{St67}, Avez~\cite{Av}, Gerl~\cite{Gerl78} (see also the recent note of Woess~\cite{Woess22})
 for random walks on groups in 
 a dynamical systems set up. Namely, we are interested in ratio limits for walks on groups driven by Gibbs Markov (GM) maps.

We recall the ratio limit in~\cite{St67, Av,Gerl78} for random walks on groups as relevant to the generalization we are after. 
Let $G$ be a group and let $X_1,\ldots, X_n$ be independent random variables on some probability space 
$(\Omega,\pP)$ taking values in $G$ distributed according to a probability measure $\nu$ on $G$. That is, given a set $A$ in $G$, $\pP(X_i\in A)=\nu(A)$.
Write $S_n=X_1\cdot\ldots \cdot X_n$, where $\cdot $ is the law of the group.

We first recall the result in~\cite{St67}. Assume that $G$ is also locally compact abelian (LCA), second countable and generated by a compact neighborhood of the identity. Let $m_G$ be the Haar measure on $G$. Under the assumption that $\nu$ is aperiodic and centered, the limit of quotients of convolutions satisfies
\begin{align}\label{eq:abiid}
\lim_{n\to\infty}\frac{\nu^{(n)}(g_0+A)}{\nu^{(n)}(g_1+A)} =\lim_{n\to\infty}
\frac{\pP(\{\omega: S_n(\omega)\in g_0+A\})}{\pP(\{\omega: S_n(\omega)\in g_1+A\})}= \frac{m_G(A+g_0)}{m_G(A+g_1)}=1,  
\end{align}
for any open and bounded set $A$ in $G$ with $m_G(A)>0$ and with
$m_G(\partial A)=0$, uniformly for $g_0,g_1$ in compact sets of $G$.
In terms of different sets in $G$, it is shown in~\cite{St67} that given  open and bounded sets $E,A$ in $G$ with $m_G(E), m_G(A)>0$ and with
$m_G(\partial E)=m_G(\partial A)=0$, $\lim_{n\to\infty}
\frac{\pP(\{\omega: S_n(\omega)\in E\})}{\pP(\{\omega: S_n(\omega)\in A\})}= \frac{m_G(E)}{m_G(A)}$.

In the deterministic set up of dynamical systems we cannot speak of  a notion similar to the convolution $\nu^{(n)}$ in \eqref{eq:abiid}. Nonetheless, as recalled below and 
as customary in the literature of dynamical systems (in particular, GM maps) we have a natural analogue of $\pP(\{\omega: S_n(\omega)\in A\})$.

For an amenable finitely generated group $G$, under a further symmetry assumption on the measure $\nu$, it is shown in~\cite{Av} that
$\lim_{n\to\infty}
\frac{\pP(\{\omega: S_n(\omega)=g\})}{\pP(\{\omega: S_n(\omega)=e\})}=1$ for all $g\in G$
and for $e$ the identity of $G$. A few years later,
 still in the set up of countable discrete groups,~\cite{Gerl78} obtained the precise limit of $\lim_{n\to\infty}
\frac{\pP(\{\omega: S_n(\omega)=g\})}{\pP(\{\omega: S_n(\omega)=e\})}$ without any symmetry condition (for not necessarily amenable groups).
We also refer to the recent note \cite{DoSh24} for a generalization of Avez's result in the non-symmetric setting.

We consider  group extensions of mixing GM maps $T:X\to X$.
The groups we shall consider will be  either discrete countable groups (which we refer to as discrete) or metrizable locally compact uncountable topological groups (which we refer to continuous).
Let $T_\psi:X\times G\to X\times G$ be the group extension of $T$ by the cocycle $\psi:X\to G$ defined via
\begin{align*}
 T_\psi(x, g)=(T x, \psi(x) g).
\end{align*}
This map  preserves the measure $\hat\mu=\mu\otimes m_G$, where  $m_G$ denotes the (left) Haar measure on  $G$. For $n\in\N$ and $x\in X$, the analogue of the i.i.d. product $S_n$ mentioned before is
\begin{align}\label{eq:defpsin}
    \psi_n(x)=\psi(T^{n-1}(x))\cdot\ldots \cdot\psi(T(x))\psi(x),
\end{align}
where $\cdot$ stands for the law of the group. As customary in the literature,
see Aaronson and Denker~\cite{AD01, AD02}, the assumption of aperiodicity and symmetry of the measure $\nu$ (in fact, any assumption on $\nu$) in the i.i.d. set up
are phrased in terms of the cocycle $\psi$.

If $G$ is continuous, $g\in G$ and $E\subseteq G$ is a `suitable' bounded set, 
we are interested in the limit
\begin{align}\label{c}
 \lim_{n\to\infty} \dfrac{\mu(\{x\in X: \psi_{n+1}(x)\in Eg \})}{\mu(\{x\in X: \psi_n(x) \in E g\})}.
\end{align}

If $G$ is discrete and $g\in G$,
we are interested in the limit
\begin{align}\label{d}
 \lim_{n\to\infty} \dfrac{\mu(\{x\in X: \psi_{n+1}(x)=g \})}{\mu(\{x\in X: \psi_n(x)=g \})}.
\end{align}

Under certain abstract assumptions listed in Section~\ref{sec:absres}, we obtain that the limit in~\eqref{c} and~\eqref{d} is exactly $1$: see Theorem~\ref{thm:amen}
in Section~\ref{sec:absres}. The proof of this result is provided in Section~\ref{sec:pf}. 
Under the same abstract assumptions, the type of arguments used in understanding the limits in~\eqref{c} (or~\eqref{d}), will allow us to also obtain an analogous form of~\eqref{eq:abiid} (or its discrete version) in the current deterministic dynamical set up. This is the content of Theorem~\ref{thm:dc}   and its proof is provided in Section~\ref{sec:dc}.\\

One feature of the abstract assumptions (listed in 
Section~\ref{sec:absres}) 
is that they are checkable.
In Section~\ref{sec:verif} we verify them  for certain groups: finite groups and certain LCA groups
(finitely generated or compactly generated).
The main tool for the verification of the abstract assumptions for the  considered LCA  is harmonic analysis and as such, transfer operators along with their perturbed version (by the cocycle $\psi$).

Although in this work, we resume to checking the abstract assumptions  listed in Section~\ref{sec:absres}  for abelian and for finite groups, we believe that these assumptions can be checked for much larger classes of groups, in particular non abelian and non compact groups. This requires a heavy use of unitary representation theory for non abelian groups as in~\cite{BaHa19, Fol} and we postpone this to a future work.
\\

In  Section~\ref{sec:ap} of the current work we combine the strategy of Stadlbauer~\cite{Ma13} and the non-symmetric Kesten criterion of Dougall and Sharp~\cite{DoSh24} to obtain an expression for the pressure function of amenable, finitely generated group extensions of GM maps.
 Under a certain symmetry condition on the cocycle $\psi$,~\cite[Theorem 4.1]{Ma13} shows that the pressure function of $T_\psi$ is equal to the pressure function of $T$. This result holds for general countable amenable groups and is obtained (among other arguments) via the Kesten criteria for symmetric random walks. 
 Without any symmetry assumption, ~\cite[Theorem 1.1]{DoSh24} gives a non-symmetric Kesten criteria for amenable finitely generated groups, which  relates the spectral radius of the non-symmetric random walk to the spectral radius of the random walk on the abelianized group.
 Combining these two results we obtain an expression for the pressure function for $T_\psi$, where $T$ is a GM map; in particular, we relate the pressure function of $T_\psi$, defined on the finitely generated group $G$, to the pressure function of the  extension of the abelianization of $G$.  
 We recall that this type of result was first obtained by Dougall and Sharp in~\cite{DoSh21} for \emph{general} group extensions of \emph{shifts of finite type} (and more generally for group extensions of Anosov flows).

\section{Statement of the main results}\label{sec:absres}

We start with a brief recall of the background of Gibbs-Markov (GM) maps.

\subsection{GM maps}\label{subsec:GM}
Roughly speaking, GM maps are infinite branch uniformly expanding maps with bounded distortion and big images. 
Throughout we assume that $(X,T,\mu,\alpha)$ is an almost onto
GM, $T$ is $\mu$-mixing and with full branches.

We recall the definitions we shall use in the sequel in more detail.
Let $(X,\mu)$ be a probability space, and let $T:X\to X$ be a topologically mixing ergodic measure-preserving transformation, piecewise continuous w.r.t. a countable partition $\alpha=\{a\}$.
For each $n\in\mathbb{N}$, we refer to $\alpha_n$ as the refined partition of the $n$-th iterate of $T$.
Define $s(x,x')$ to be the least integer $n\ge0$ such that $T^nx$ and $T^nx'$ lie in distinct partition elements.
Since $T$ is expanding, $s(x,x')=\infty$ if and only if $x=x'$, one obtains that $d_\beta(x,x')=\beta^{s(x,x')}$
for $\beta\in(0,1)$ is a metric.

Take the potential $\varphi=\log\frac{d\mu}{d\mu\circ T}:X\to\R$.
We say that $T$ is a GM map if the following hold w.r.t. the countable partition $\alpha$:
\begin{itemize}

\parskip = -2pt
\item For each $a\in\alpha$, $T (a)$ is a union of partition elements and $T|_a:a\to T(a)$ is a measurable bijection for each $j\ge1$
such that $T(a)$ is the union of elements of the partition $\bmod \mu$;
\item $\inf_{a\in\alpha}\mu(T(a))>0$. This is usually referred to as the big image property (BIP).
\item
There are constants $C>0$, $\beta\in(0,1)$ such that
$|\varphi(x)-\varphi(x')|\le Cd_\beta(x,x')$ for all $x,x'\in a$ and $a\in\alpha$, $j\ge1$.
\end{itemize} 
 The map $T$ has full branches if $T(a)=X$ for each $a\in\alpha$.
The invariant measure $\mu$ is known
to have the Gibbs property. That is,  there exists $C>0$ such that for every $n\in\mathbb{N}$ and for all $a\in \alpha_n$, $x\in a$,
\begin{align}\label{Eq: Gibbs prop}
    C^{-1}\mu(a)\le e^{\varphi_n(x)}\le C\mu(a).
\end{align}
See, for instance,~\cite[Chapter 4]{Aaronson} and~\cite{AD01}
for a more detailed exposition on Gibbs-Markov maps.  

\subsection{Main assumptions and statement of the abstract result}
The standing assumption for the main abstract result are equation~\eqref{eq:Cond-psi}
or~\eqref{eq:Cond-psic} below,
which we verify for certain groups in Section~\ref{sec:verif}.

Throughout, we will use the following notation. For every $d, A, B>0$, we denote
$A \sim_{d} B
\text{ when }
\frac{B}{d} \leq A \leq d B$.
Across this text, $\psi_n$ is as defined in~\eqref{eq:defpsin} and $e\in G$ will be the identity element.

We assume one of the following, depending on whether $G$ is discrete or continuous.

\begin{description}[style=multiline,labelwidth=3cm,align=parright, leftmargin=0pt, itemindent=25pt]
\item[(D)]\makeatletter\def\@currentlabel{D}\makeatother\label{eq:Cond-psi}
If $G$ discrete,
given $g \in G$,
we assume that $\forall \eps > 0\, \forall n_0 \in \N\, \exists n_1 > n_0 \, \forall n \geq n_1 \, \exists n' \in \{n_0, \dots, n_1\} \forall b \in \alpha_{n'}$:
\begin{equation*}
 \frac{ \mu(b \cap \{ x \in X : \psi_n(x)= g \} )}{\mu(b)} \sim_{1+\eps} \mu( \{ x \in X : \psi_n(x)=g \} ).
\end{equation*}
\item[(C)]\makeatletter\def\@currentlabel{C}\makeatother\label{eq:Cond-psic}
If $G$ is continuous,
let $E$ be a bounded set containing $e$ such that $m_G(E) = m_G(E \setminus \partial E) > 0$.
Given $g \in G$, we assume that
$\forall \eps > 0\, \forall n_0 \in \N\, \exists n_1 > n_0 \, \forall n \geq n_1 \, \exists n' \in \{n_0, \dots, n_1\} \forall b \in \alpha_{n'}$:
\begin{equation*}
 \frac{ \mu(b \cap \{ x \in X : \psi_n(x)\in E g \} )}{\mu(b)} \sim_{1+\eps} \mu( \{ x \in X : \psi_n(x)\in E g\} ).
\end{equation*}
\end{description}
In Section~\ref{sec:verif}, we verify~\eqref{eq:Cond-psi} and~\eqref{eq:Cond-psic} (for finite and certain abelian groups) under the symmetry assumption phrased below as \eqref{eq:S}.
This type of symmetry assumption is crucial in establishing that, when $G$ is continuous, $\mu( \{ x \in X : \psi_n(x)\in E g\} )>0$ for $n$ sufficiently large: see Lemma~\ref{lemma:psitive mu(e)} below.
The symmetry condition on $\psi$ that we require is
\begin{description}[style=multiline,labelwidth=3cm,align=parright, leftmargin=0pt, itemindent=23pt]
 \item[(S)]\makeatletter\def\@currentlabel{S}\makeatother\label{eq:S}
There exists an invertible probability   measure preserving transformation $S:X\to X$ so that $S$ and $S^{-1}$ are  continuous and 
\begin{align*}
 S\circ T=T\circ S\text{ and }(\psi\circ S)(x)=\psi(x)^{-1}, \text{ for }x\in X.
\end{align*}
\end{description}
Note that this symmetry assumption~\eqref{eq:S} implies, in particular, that 
$\mu(\{x\in X:\psi(x)\in E\})=\mu(\{x\in X:\psi(x)^{-1}\in E\})$, for $E\subseteq  G$, which is the usual assumption in an independent set up. In Section~\ref{sec:verif}, we will use a slightly weaker symmetry assumption: see text around equation~\eqref{symmetryab} as used in~\cite{AD02}.

If $G$ is discrete, we assume that the cocycle $\psi:X\to G$ is constant on each $a\in\alpha$.  That is, for every $a\in \alpha$, $\psi(x)=\psi(y)$ for each $x,y\in a$. In the discrete case, this assumption on $\psi$ will be used
throughout the entire work.

When $G$ is continuous we assume that $\psi$ is continuous.
 In Section~\ref{sec:pf} we will  use the continuity of $\psi$ in the proofs of Lemmas~\ref{lemma:psitive mu(e)} and~\ref{lem:hd} below.
(In Section~\ref{sec:verif}, we will use a stronger form of continuity, namely H\"older on the partition elements $a\in\alpha$).\\

When $G$ is continuous we need a further assumption.
\begin{description}[style=multiline,labelwidth=3cm,align=parright, leftmargin=0pt, itemindent=35pt]
 \item[(CM)]\makeatletter\def\@currentlabel{CM}\makeatother\label{eq:contmix}
 Let $E,A$ be bounded sets in $G$ with $m_G(\partial E)= m_G(\partial A)=0$.
 
Let $F$ be a bounded set in $G$ with $m_G(F) \in (0,\infty)$ and fix $k \in \N$.
We assume that for any $a\in\alpha_k$ and for all $n$ large 
\begin{align*}
\mu\otimes m_G \left((a \times F)\cap T_\psi^{-n}(X\times A)\right)
 \le \mu(a) m_G(F) \frac{m_G(A)}{m_G(E)}\mu( \{ x \in X : \psi_n(x)\in E g\} ).
\end{align*}
\end{description}

In Section~\ref{sec:verif} we verify~\eqref{eq:Cond-psic} and obtain a much stronger form of~\eqref{eq:contmix} for compactly generated LCA groups (see also Proposition~\ref{prop:comp} below). But in the proof of the abstract results below we only need  the current form of~\eqref{eq:contmix}.\\

With all the assumptions specified, we state the abstract results.

\begin{theorem}\label{thm:amen}
Let $T:X\to X$ be a GM map. Let $G$ be a topological group and assume that $e\in \psi(X)$.
\begin{itemize}
 \item[(a)] 
If $G$ is discrete, suppose that~\eqref{eq:Cond-psi} holds.
Then for each $g\in\cup_{k\in\N}\psi_k(X)$,
\begin{align*}
\lim_{n\to\infty}    \dfrac{\mu(\{x\in X: \psi_{n+1}(x)=g \})}{\mu(\{x\in X: \psi_n(x)=g \})}=1.
\end{align*}

\item[(b)] 
If $G$ is continuous, we assume that~\eqref{eq:S},~\eqref{eq:Cond-psic} and~\eqref{eq:contmix} hold.
Then
\begin{align*}
\lim_{n\to\infty}    \dfrac{\mu(\{x\in X: \psi_{n+1}(x)\in Eg \})}{\mu(\{x\in X: \psi_n(x) \in Eg \})}=1,
\end{align*}
for any $g\in \cup_{k\in\mathbb{N}}\psi_k(X)$, and any open bounded set $E\subset G$ with $e \in E$ and $m_G(E) = m_G(E \setminus \partial E)>0$.
\end{itemize}
\end{theorem}

\begin{theorem}\label{thm:dc}
Assume the set up of Theorem~\ref{thm:amen}.
\begin{itemize}
 \item[(a)] 
If $G$ is discrete, for $g\in\cup_{k\in\N}\psi_k(X)$,
\begin{align*}
\lim_{n\to\infty}    \dfrac{\mu(\{x\in X: \psi_{n}(x)=g \})}{\mu(\{x\in X: \psi_n(x)=e \})}=1.
\end{align*}

\item[(b)] 
If $G$ is continuous, let $E$ be a bounded set in $G$ with $m_G(E) = m_G(E \setminus \partial E)>0$. Then uniformly for $g, g_1\in\cup_{k\in\N}\psi_k(X)$,
\begin{align*}
\lim_{n\to\infty}    \dfrac{\mu(\{x\in X: \psi_{n}(x)\in Eg \})}{\mu(\{x\in X: \psi_n(x) \in Eg g_1^{-1}\})}=1.
\end{align*}
\end{itemize}
\end{theorem}

If $G$ is LCA, second countable, and compactly generated, and $\psi$ is symmetric in the sense of~\eqref{eq:S}, in the process of verifying~\eqref{eq:Cond-psic}, we  obtain directly an analogue of Stone's result for different sets: 
\begin{proposition}\label{prop:comp}
 Given that $E,A$ are bounded sets in $G$ with $m_G(\partial E)= m_G(\partial A)=0$,
\begin{equation*}
 \frac{ \mu( \{ x \in X : \psi_n(x)\in E  \} )}{\mu( \{ x \in X : \psi_n(x)\in A  \} )} \to \frac{m_G(E)}{m_G(A)},\quad \text{ as } n \to \infty.
\end{equation*}
\end{proposition}
The proof of Proposition~\ref{prop:comp} is provided in subsection~\ref{subsec:ab}.

\section{Proof of Theorem \ref{thm:amen} under \eqref{eq:Cond-psi} or \eqref{eq:Cond-psic} and \eqref{eq:contmix} }\label{sec:pf}
Recall that the elements in $\alpha_n$
are $n$-cylinders. 
For simplicity of notation, throughout this section we work with $\mu^n(g)$
and $\mu^n(Eg)$ as defined below.
When $G$ is discrete, for $g\in G$, we define
\begin{align*}
   \mu^n(g) := \mu(\{x\in X:\psi_n(x)=g\})=\sum_{a\in\alpha_n, \psi_n(a)=g} \mu(a).
\end{align*}
 If $G$ is a continuous group, for $g\in G$ and $E\subseteq G$ a bounded set that contains $e\in G$, we write
\begin{align*}
   \mu^n(Eg) := \mu(\{x\in X: \psi_n(x)\in E g\})=\sum_{a\in\alpha_n, \psi_n(a) \cap E g \neq \emptyset} \mu(a \cap \psi_n^{inv}(E g)),
\end{align*}
where $\psi_n^{inv}$ denotes the inverse image of the map $\psi_n:X \to G$
(as opposed to the inverse element $\psi(x)^{-1}$ of $\psi(x)\in G$).

 Equipped with this notation, we define the following sequence of measures
 on $X$.
For $g\in G$, let
\begin{align}\label{eq:mn}
  \int v \, dm_n^g :=\dfrac{1}{\mu^n(g)}\sum_{a\in\alpha_n, \psi_n(a)=g}\int_a  v \, d\mu, \quad \text{ if } G\text{ is discrete}.
\end{align}
Given $g\in G$ and $E \subset G$ a bounded set containing $e \in G$ such that
$m_G(E) = m_G(E \setminus \partial E)>0$,
\begin{align}\label{eq:mnc}
  \int v \, dm_n^{Eg} :=&\dfrac{1}{\mu^n(Eg)} \!\!  \sum_{\stackrel{a\in\alpha_n}{\psi_n(a)\cap Eg \neq \emptyset}}\!\!\!\!\!\!   \frac{\mu(a \cap \psi_n^{inv}(Eg))}{\mu(a)} \int_a  v \, d\mu, \quad \text{ if } G\text{ is continuous.}
\end{align}

To ensure that the sequence of measures just defined has a chance to behave well as $n\to\infty$, we will first check that  $\mu^n(g)>0$, or $\mu^n(Eg)>0$  as in Lemmas~\ref{lemma: positivemun} and~\ref{lemma:psitive mu(e)} below.

\begin{lemma}\label{lemma: positivemun}
    If $G$ is a discrete group and $e\in \psi(X)$, then $\mu^n(e)>0$ for every $n\in\mathbb{N}$.
    Moreover, if for $g\in G$ there exists $r_g\in\mathbb{N}$ such that $g\in \psi_{r_g}(X)$, then $\mu^n(g)>0$ for every $n\geq r_g$. 
\end{lemma}
\begin{proof}
Recall that $T$ has full branches.    Since $\psi$ is constant on partition elements, there is $a   \in \alpha$ and a fixed point $x \in a$ such that $\psi(x) = \psi(a) = e$.
Then for the $n$-cylinder $a_n \owns x$, we get $\psi_n(a_n) = e$, and because $\mu$ is fully supported, $\mu^n(e) \geq \mu(a_n) > 0$.
For the second part of this lemma, let $g\in G$ and $r_g\in \mathbb{N}$ as stated before. 
By the previous part, for every $m\geq 1$, there exists $c_m\in \alpha_m$ such that $\psi_m(c_m)=e$.
Then, $\psi_{r_g+m}(T^{-m}_{c_m}a_{r_g})=g$ and we deduce that $\mu^{r_g+m}(g)>0$.
\end{proof}

Recall that when $G$ is continuous $\psi$ is continuous. The following lemma generalizes the continuity property to $\psi_n$ on cylinders.

\begin{lemma}\label{lemma:contpsi}
Let $G$ be a continuous group.    For every $n\in\mathbb{N}$, the map $\psi_n:X\to G$ is continuous in $n$-cylinders.
\end{lemma}
\begin{proof}
    Let $n\in\mathbb{N}$.
    Let $x\in X$ and $\varepsilon>0$.
    Let $d_G$ be a right invariant metric of $G$ that generates the topology of $G$. 
    Observe that for $y\in X$, 
    \begin{align*}
        d_G(\psi_n(x),\psi_n(y))\leq
        & d_G(\psi(T^{n-1}(x))\cdots \psi(T(x))\psi(x),\psi(T^{n-1}(x))\cdots \psi(T(x))\psi(y))\\
        &+ d_G(\psi(T^{n-1}(x))\cdots \psi(T(x))\psi(y),\psi(T^{n-1}(x))\cdots \psi(T(y))\psi(y))\\
        &+ \cdots\\
        &+ d_G(\psi(T^{n-1}(x))\psi(T^{n-2}(y))\cdots \psi(T(y))\psi(y),\psi(T^{n-1}(y))\cdots \psi(T(y))\psi(y))
    \end{align*}
    Using the right invariance of $d_G$, we obtain that
    \begin{align*}
        d_G(\psi_n(x),\psi_n(y))\leq
        & d_G(\psi(T^{n-1}(x))\cdots \psi(T(x))\psi(x),\psi(T^{n-1}(x))\cdots \psi(T(x))\psi(y))\\
        &+ d_G(\psi(T^{n-1}(x))\cdots \psi(T(x)),\psi(T^{n-1}(x))\cdots \psi(T(y)))\\
        &+ \cdots\\
        &+ d_G(\psi(T^{n-1}(x)),\psi(T^{n-1}(y)))
    \end{align*}
    
    Next, note that for every $z\in G$, the map $L_z:G\to G$ given by $L_z(g)=zg$ is continuous.
    This fact  guarantees that there exists $\delta_1$ such that if $d_G(\psi(T^i(x)),\psi(T^i(y)))<\delta_1$, then  $$d_G(\psi(T^{n-1}(x))\cdots\psi(T^{i+1}(x))\psi(T^i(x)),\psi(T^{n-1}(x))\cdots\psi(T^{i+1}(x))\psi(T^i(y)))<\frac{\varepsilon}{n}$$
for each $i\in \{0,\ldots, n-1\}$.
Now, as $\psi$ and $T$ are continuous, there exists $\delta>0$ such that if $d_X(x,y)<\delta$, then $d(\psi(T^i(x)),\psi(T^i(y)))<\delta_1$ for each $i\in\{0,1,\ldots, n-1\}$.
We conclude the proof.
\end{proof}

Using Lemma~\ref{lemma:contpsi}, we obtain the needed condition in the continuous case.
\begin{lemma}\label{lemma:psitive mu(e)}
Let $G$ be a continuous group. Assume~\eqref{eq:Cond-psic}  and~\eqref{eq:S}. Let $A$ be a set in $G$ such that $A$ satisfies the conditions of~\eqref{eq:Cond-psic}.
Let $g\in \bigcup_{i\in\mathbb{N}}\psi_i(X)$.
Then, there exists $n_g\in \mathbb{N}$ such that $\mu^n(Ag)>0$ for every $n \geq n_g$ and every non-empty open neighborhood $A$ of $e$.
\end{lemma}
\begin{proof}Throughout the proof, we will use that, by Lemma~\ref{lemma:contpsi},
$\psi_n$ is continuous in $n$-cylinders.

    We first  use \eqref{eq:S}.
    Take $x_0\in X$ such that $T(x_0)=S(x_0)$. 
    Since $S^{-1}\circ T: a\to X$ is onto and continuous for each $a\in\alpha$, there exists a point $x_0\in a$ such that $S^{-1}\circ T(x_0)=x_0$, so $T(x_0)=S(x_0)$.
    For such a point, $T^2(x_0)=T\circ S(x_0)=S\circ T (x_0)=S^2(x_0)$ and $\psi_2(x_0)=\psi(T(x_0))\psi(x_0)=\psi(S(x_0))\psi(x_0)=e$.
    As $\psi_{2n}$ is continuous, then $\psi_{2n}^{inv}(A)$ is a non-empty open set of $X$ for each open set $A$ containing $e$. 
    Hence, $\mu(\psi_{2n}^{inv}(A))>0$, for each $n\in\mathbb{N}$.
    
  Now, let $A$ be an open neighborhood of $e$.
    By the continuity of the operation in $G$, there exists $A_1\subseteq G$, an open neighborhood of $e$, such that $A_1A_1\subseteq A$. 
    Note that  $\psi^{inv}(A_1)$ is a non-empty open set.
    Thus, there are $n_0\in\mathbb{N}$, $a_n\in \alpha_n$, for each $n\geq n_0$, such that $a_n\subseteq \psi^{inv}(A_1)$. 
    Applying condition \eqref{eq:Cond-psic} for any $\varepsilon>0$ and for $n_0$ as above, there exist $n_1\geq n_0$ and $n'\in \{n_0,\ldots,n_1\}$ satisfying \eqref{eq:Cond-psic}.
    Let $n\geq n_1$ be an even number.
    Consider $b\in \alpha_{n'+1}$ in such a way that $b\subseteq \psi^{inv}(A_1)$.
    Thus, there exists $c\in \alpha_{n'}$ such that $T(b)=c$.
    Condition \eqref{eq:Cond-psic} implies that $\mu(c\cap \psi_{n}^{inv}(A_1))>0$ for each $n\geq n_1$.
    As $\mu$ is invariant and it has full support, $\mu(b\cap T^{-1}(\psi^{inv}_n(A_1)))>0$.
    In particular, there exists $x\in b\cap T^{-1}(\psi^{inv}_n(A_1))$ and consequently, $\psi_n(T(x))\psi(x)=\psi_{n+1}(x)\in A_1A_1\subseteq A$.
    Using that $\psi_{n+1}$ is continuous, we obtain that $\psi^{inv}_{n+1}(A)$ is a non-empty open set, and hence, $\mu(\psi^{inv}_{n+1}(A))>0$. 
    Therefore, if $n_e=n_1+1$ then $\mu^n(\psi_n^{inv}(A))>0$ for each $n\geq n_e$.
The proof for $g\in \bigcup_{i\in\mathbb{N}}\psi_i(X)$ works in an analogous way.
\end{proof}

The potential $\varphi$ is H\"older continuous, so there is $\beta \in (0,1)$ and $C_H > 0$ satisfying
$|\varphi(x)-\varphi(y)| \leq C_H d_X(x,y)^\beta$
for all $x,y \in X$. Cylinder sets are exponentially small, so we can find $\lambda_0 = \lambda_0(T, \beta) \in (0,1)$ such that
\begin{equation}\label{eq:phi-holder}
 e^{\varphi_k(x) } \sim_{1+\lambda_0^{n-k}} e^{\varphi_k(y)} \quad \text{ for all } x,y \in a \in \alpha_n,\, n > k,
\end{equation}
where $\varphi_k(x) = \sum_{j=0}^{k-1} \varphi(T^j(x))$ is the $k$-th ergodic sum.
Also, due to the intermediate value theorem for integrals, for each $k \in \N$, $m > k$, and every measurable $b \subset a \in \alpha_m$, we can select $\xi^k_b \in a$ such that
\begin{equation}\label{eq:IMT}
 \mu(T^k b) = \int_b e^{ -\varphi_k(x) } d\mu(x)
 = e^{ -\varphi_k(\xi^k_b) } \mu(b).
\end{equation}

\begin{lemma}\label{lem:m_nu}
 Assume \eqref{eq:Cond-psi} or~\eqref{eq:Cond-psic}
 and let $u \in \alpha_k$ for some $k$.
 Then for each $g \in G$, we have $m_n^g(u) \to \mu(u)$ and $m_n^{Eg}(u)\to \mu(u)$ as $n \to \infty$.
\end{lemma}

\begin{proof}
We provide the argument assuming~\eqref{eq:Cond-psic}.
The argument assuming~\eqref{eq:Cond-psi} is easier, and also a special case of this,
because when taking $Eg = \{ g \}$ in the discrete case, we have
$\mu(a \cap \psi_n^{inv}(Eg)) = \mu(a)$ if
$\psi_n(a) = g$ and $\mu(a \cap \psi_n^{inv}(Eg)) = 0$ otherwise. We recall that in the discrete case, we assume that $\psi$ is constant on each $a\in\alpha$.

\medskip

 Let $\eps> 0$ and $n_0 \geq k$ be arbitrary and $n_1 > n_0$ such that \eqref{eq:Cond-psic} holds.
 Recall the definition of $m_n^{Eg}$ in~\eqref{eq:mnc}.
 Then for $n > n_1$ and $n_0 \leq n' \leq n_1$ as in  \eqref{eq:Cond-psic}, we have
 \begin{eqnarray*}
  m_n^{Eg}(u) &=& \frac{1}{\mu^n(Eg)} \sum_{\stackrel{a \in \alpha_n, a \subset u}{ \psi_n(a) \cap Eg \neq \emptyset}} \mu(a \cap \psi_n^{inv}(Eg)) \\
   &=& \frac{1}{\mu^n(Eg)} \sum_{\stackrel{b \in \alpha_{n'}}{b \subset u}}
   \sum_{\stackrel{a \in \alpha_n, a \subset b}{\psi_n(a) \cap Eg \neq \emptyset}} \mu(a \cap \psi_n^{inv}(Eg)) \\
    &=& \frac{1}{\mu^n(Eg)} \sum_{\stackrel{b \in \alpha_{n'}}{b \subset u}}
            \mu(b \cap \psi_n^{inv}(Eg)) \sim_{1+\eps} \sum_{b \in \alpha_{n'}, b \subset u} \mu(b) = \mu(u),
 \end{eqnarray*}
where in the last line we used \eqref{eq:Cond-psic}.
Since $\eps > 0$ is arbitrary, the result follows.
\end{proof}

First, we record a lemma, which is a main ingredient
for the proof of Theorem~\ref{thm:amen}, when $G$ is discrete.

\begin{lemma}\label{firstnallemma61ds}Suppose that $G$ is discrete.
 Let $u\in\alpha_k$ for some $k \in \N$ and assume \eqref{eq:Cond-psi}.
 Then for every $g\in\cup_{k\in\N}\psi_k(X)$,
 \[
 \lim_{n\to\infty} \left|m_n^g(u)-\frac{ \mu^{n-k}(g\psi_k(u)^{-1}) }{\mu^n(g)} \, \mu(u) \right| = 0.
 \]
\end{lemma}

\begin{proof}Recall the definition of $m_n^{g}$ in~\eqref{eq:mn}.
For $g\in\cup_{k\in\N}\psi_k(X)$, and the cylinder set $u\in\alpha_k$, we have
\begin{eqnarray*}
    m_n^g(u)
     &=& 
     \dfrac{1}{\mu^n(g)} \sum_{\stackrel{a\in\alpha_n, \psi_n(a)=g}{a \subset u}} \mu(a) \\
    &=& \dfrac{ \mu^{n-k}(g\psi_k(u)^{-1}) }{ \mu^n(g) } \dfrac{1}{ \mu^{n-k}(g\psi_k(u)^{-1}) }  \sum_{\stackrel{a \subset u, T^k(a) = a' \in \alpha_{n-k}}{\psi_{n-k}(a')=g \psi_k(u)^{-1}}} \frac{\mu(a)}{\mu(a')} \mu(a')
    \\
    &=& \dfrac{ \mu^{n-k}(g\psi_k(u)^{-1}) }{ \mu^n(g) } \dfrac{1}{ \mu^{n-k}(g\psi_k(u)^{-1}) }  \sum_{\stackrel{a \subset u, a' \in \alpha_{n-k}}{\psi_{n-k}(a')=g \psi_k(u)^{-1}}} e^{\varphi_k(\xi^k_a)} \mu(a'),
\end{eqnarray*}
where in the last line we used \eqref{eq:IMT} and the fact that $T^k:u \to X$ is a bijection.
Choose $n \ge n-k > n_1 \geq n' > n_0$ according to \eqref{eq:Cond-psi}
where $n_0$ is such that $\lambda^{n_0-k} < \eps$.
Then, using H\"older continuity of $\varphi$ as in \eqref{eq:phi-holder}, we estimate
\begin{align*}
 \dfrac{1}{ \mu^{n-k}(g\psi_k(u)^{-1}) } & \sum_{\stackrel{a' \in \alpha_{n-k}}{\psi_{n-k}(a')=g\psi_k(u)^{-1}}} e^{\varphi_k(\xi^k_a)} \mu(a') \\
 &=\
 \dfrac{1}{ \mu^{n-k}(g\psi_k(u)^{-1}) }  \sum_{b \in \alpha_{n'}} \sum_{\stackrel{a' \in \alpha_{n-k}, a' \subset b}{\psi_{n-k}(a')=\psi_k(u)^{-1}}} e^{\varphi_k(\xi^k_a)} \mu(a') \\
& \sim_{1+\lambda^{n'-k}}\
 \dfrac{1}{ \mu^{n-k}(g\psi_k(u)^{-1}) }  \sum_{b \in \alpha_{n'}}  e^{\varphi_k(\xi^k_b)} \sum_{\stackrel{a' \in \alpha_{n-k}, a' \subset b}{\psi_{n-k}(a')=g\psi_k(u)^{-1}}} \mu(a') \\
\text{by (D)} & \sim_{1+\eps}\
 \sum_{b \in \alpha_{n'}}  e^{\varphi_k(\xi^k_b)}\mu(b) \\
& \sim_{1+\lambda^m}\  \sum_{b \in \alpha_{n'}}  \int_b e^{\varphi_k(T^{-k}_ux)}\, d\mu(x) \ =\  \int_X e^{\varphi_k(T^{-k}_ux)}\, 1_u  d\mu(x) \\
& =\  \int_X L^k 1_u \, d\mu\ =\ \mu(u),
\end{align*}
where $T^{-k}_u$ denotes the inverse branch of $T^k:u \to X$.
In conclusion,
$$
m_n^g(u) \sim_{(1+\eps)^3}
\dfrac{ \mu^{n-k}(g\psi_k(u)^{-1}) }{ \mu^n(g) }  \mu(u).
$$
Since $\eps > 0$ was arbitrary and $m_n^g(u) \leq 1$, the result follows.
\end{proof}

Next, we record a lemma, which is a main ingredient
for the proof of Theorem~\ref{thm:amen} when $G$ is continuous.

\begin{lemma}\label{lem:hd} Let $G$ be a continuous topological group and assume that
~\eqref{eq:contmix} holds for some compact set $F \subset G$ with $m_G(F) \in (0,\infty)$.
Fix a bounded set $E \subset G$ with $m_G(E) = m_G(E \setminus \partial E) > 0$.
Choose $x_0 \in X$ and $r \in \N$ arbitrary, and set $g_1 := \psi_r(x_0)$.
Then for every $\eps > 0$, there is $n_0 \in \N$ with the following properties.
Let $u' \in \alpha_{n_0+r}$ be the $n_0+r$-cylinder containing $x_0$ and $u = T^r(u')$,
and let $T_{u'}^{-r}:u \to u'$ denote the inverse branch of $T^r$ mapping to $u'$.
Then there is $n_2  > n_0$ such that
 \begin{align*}\label{eq:hd}
  \frac{ \left| \mu(\{ x \in u : \psi_n(x) \in Eg \} ) - \mu(\{ x \in u  : \psi_n(x) \in Egg_1\psi_r(T^{-r}_{u'}x)^{-1} \} ) \right| }
 {\mu(\{ x \in u : \psi_n(x) \in Eg\})}   < \eps.
\end{align*}
for every $n \geq n_2$.
\end{lemma}

\begin{remark}\label{rem:hd}
From the proof of this lemma, it is also clear that for $n_1 > n_0$ we can adjust $n_2 > n_1$ such that the statement holds for every
 $n'$-cylinder $a \subset u$ for $n_0 \leq n' \leq n_1$ instead of $u$ itself.
\end{remark}

\begin{proof}
We will exploit assumption~\eqref{eq:contmix} with $F = \overline{B_\eta(e)}$, a closed $\eta$-ball around the identity, with $\eta$ still to be determined.

Choose $\eps > 0$ arbitrary, and suppose $Eg \subset G$ and $x_0\in X$ with $g_1 = \psi_r(x_0)$ are as in the lemma.
Let $\partial_\eta Eg  := \{ h \in G : d_G(\partial Eg, h) < \eta\}$.
By translation invariance of the Haar measure, 
$m_G(\partial_\eta Eg) \to 0$ as $\eta \to 0$.
Fix $\eta > 0$ such that
\begin{equation}\label{eq:et}
 m_G(\partial_{2\eta} Eg) < \eps.
\end{equation}

By continuity of $\psi_r$, we can choose $n_0$ so large, i.e.,  the cylinder $u \in \alpha_{n_0}$ containing $T^r(x_0)$ is so small, that $d_G(g_1 \psi_r(T^{-r}_{u'}x)^{-1}, e) < \eta$ for all $x \in u$.
If $\psi_n(x) \in Eg$ but $\psi_n(x) \notin Egg_1 \psi_r(T^{-r}_{u'}x)^{-1}$
or if $\psi_n(x) \notin Eg$ but $\psi_n(x) \in Egg_1 \psi_r(T^{-r}_{u'}x)^{-1}$,
then $\psi_n(x) \in \partial_\eta Eg$.
Hence,
\begin{align*}
 \left| \mu(\{ x\in u : \psi_n(x) \in Eg\})  \right. & - \left.  \mu(\{ x\in u : \psi_n(x) \in Egg_1 T^{-r}_{u'}x)^{-1} \}) \right| \\[1mm]
 &\leq \mu(\{ x\in u : \psi_n(x) \in \partial_\eta Eg\}) \\
 &\le \mu(\{ x\in u : \psi_n(x)g' \in \partial_{2\eta} Eg \text{ for all } g' \in F\}) \\
 &\le \frac{1}{m_G(F)} \, \mu \otimes m_G(\{ (x,g') \in u \times F : T^n_\psi(x,g') \in \partial_{2\eta} Eg\})\\
 &=\frac{\mu\otimes m_G \left((u \times F)\cap T_\psi^{-n}(X\times \partial_{2\eta} Eg)\right)}{m_G(F)},
 \end{align*}
 where in the third line we have used that $F = \overline{B_\eta(e)}$
 and in the fourth line we have used that
 \begin{align*}
  \mu \otimes m_G(\{ (x,g') \in u \times F : T^n_\psi(x,g') \in \partial_{2\eta} Eg\})&=\mu \otimes m_G(\{ (x, g')\in u\times F : \psi_n(x)g' \in \partial_{2\eta} Eg)\\
  &\ge\mu(\{ x\in u : \psi_n(x)g' \in \partial_{2\eta} Eg \text{ for all } g' \in F\})  m_G(F).
 \end{align*}

 By~\eqref{eq:contmix} with $a=u$ and $A=\partial_{2\eta} Eg$,
 \begin{align*}
  \frac{1}{m_G(F)}\mu\otimes m_G \left((u \times F)\cap T_\psi^{-n}(X\times \partial_{2\eta} Eg)\right)&\le \mu(u)\frac{m_G(\partial_{2\eta} Eg)}{m_G(E)}\mu(\{x \in X : \psi_n(x) \in Eg\})\\
  &\le \eps \frac{\mu(u)}{m_G(E)}\mu(\{x \in X : \psi_n(x) \in Eg\}),
 \end{align*}
where in the last equation we have used~\eqref{eq:et}. The conclusion follows since $\eps$ is arbitrary.
\end{proof}

Using Lemma~\ref{lem:m_nu} together with Lemma~\ref{firstnallemma61ds} or Lemma~\ref{lem:hd}, we can complete

\begin{proof}[Proof of Theorem~\ref{thm:amen}]
We do the continuous case, that is item (b), because if $G$ is discrete,
then it is a special case of the continuous case (as already mentioned at the beginning of the proof of
Lemma~\ref{lem:m_nu}). In the discrete case everything goes the same,
except that the step involving the use of Lemma~\ref{lem:hd} is not required
and instead we have a much easier use of Lemma~\ref{firstnallemma61ds}.

\medskip

Throughout the proof we work with $m_n^{Eg}$ as defined in~\eqref{eq:mnc}.
By assumption, we can fix  $x_0 \in X$ such that $\psi(x_0) = e$.
Choose $\eps > 0$ and $n_1 \in \N$ arbitrary and $n_1 > n_0$ so large that
Lemma~\ref{lem:hd} holds for all $u = T(u')$, $x_0 \in u' \in \alpha_{n_0+1}$  and $n \geq n_2 > n_1$,
where $n_1$ is chosen as in~\eqref{eq:Cond-psic}, and so that Remark~\ref{rem:hd} applies as well.

Take $n \geq n_2$ so large that $\lambda_0^{n-n_2}  < \eps$ and $\mu^n(Eg) > 0$.
For $n' \in [n_0, n_1]$ as in~\eqref{eq:Cond-psic}, we get
\begin{eqnarray*}
m_n^{Eg}(u)
 &=& \dfrac{1}{\mu^n(Eg)} \sum_{a\in\alpha_n, a \subset u} \mu(a \cap  \psi_n^{inv}(Eg)) \\
 &=& \dfrac{ 1 }{ \mu^n(Eg) }  \sum_{a' \in\alpha_{n'}, a' \subset u} \mu(a' \cap  \psi_n^{inv}(Eg) )\\
\text{by~\eqref{eq:Cond-psic}} &\sim_{1+\eps}&
 \dfrac{ 1 }{ \mu^n(Eg)}  \sum_{a' \in\alpha_{n'}, a' \subset u} \mu(a') \mu(\{ x \in X : \psi_n(x) \in Eg\})\\
 &=& \dfrac{\mu(u)}{ \mu^n(Eg)}  \mu^n(Eg) = \mu(u).
\end{eqnarray*}
A similar, but more extended, computation on $u'$ gives for $n' \geq n_0$:
\begin{eqnarray*}
m_{n+1}^{Eg}(u')
 &=& \dfrac{1}{\mu^{n+1}(Eg)} \sum_{a\in\alpha_{n+1}, a \subset u'} \mu(a \cap  \psi_{n+1}^{inv}(Eg)) \\
 &=& \dfrac{ 1 }{ \mu^{n+1}(Eg) }  \sum_{\stackrel{Tb = a' \in \alpha_{n'}} {b \in \alpha_{n'+1}, b \subset u'}} \mu(b \cap  \psi_{n+1}^{inv}(Eg) )\\
 \begin{array}{c}\text{by the Gibbs}\\
  \text{property } \eqref{eq:phi-holder}
 \end{array}
 &\sim_{1+\eps}& \dfrac{ 1 }{ \mu^{n+1}(Eg) }  \sum_{\stackrel{Tb = a' \in \alpha_{n'}} {b \in \alpha_{n'+1}, b \subset u'}}
 e^{-\varphi(\xi_b^1)} \mu(\{x \in a' :  \psi_n(x) \psi(T^{-1}_b(x)) \in Eg\})\\
 \text{by~\eqref{eq:IMT}} &\sim_{1+\eps}& \dfrac{ 1 }{ \mu^{n+1}(Eg) }  \sum_{\stackrel{Tb = a' \in \alpha_{n'}} {b \in \alpha_{n'+1}, b \subset u'}}
 \frac{\mu(b)}{\mu(a')} \ \mu(\{x \in a' :  \psi_n(x) \in Eg \psi(T^{-1}_b(x))^{-1}\}  )\\
 \begin{array}{c}
\text{by Lemma~\ref{lem:hd}}\\
\text{with } g_1 = e
 \end{array}
 &\sim_{1+\eps}& \dfrac{ 1 }{ \mu^{n+1}(Eg) }  \sum_{\stackrel{Tb = a' \in \alpha_{n'}} {b \in \alpha_{n'+1}, b \subset u'}}
 \frac{\mu(b)}{\mu(a')} \ \mu(\{x \in a' :  \psi_n(x) \in Eg\})\\
 \text{by~\eqref{eq:Cond-psic}} &\sim_{1+\eps}& \dfrac{ 1 }{ \mu^{n+1}(Eg) }  \sum_{\stackrel{Tb = a' \in \alpha_{n'}} {b \in \alpha_{n'+1}, b \subset u'}} \mu(b) \ \mu(\{x \in X :  \psi_n(x) \in Eg\})\\
 &=& \dfrac{ \mu(u') }{ \mu^{n+1}(Eg) } \ \mu^n(Eg).
\end{eqnarray*}
Putting the two previous computations together, we obtain the quotient:
\[
\frac{m_{n+1}^{Eg}(u')}{m_n^{Eg}(u)} \sim_{(1+\eps)^4}
\frac{\mu(u')}{\mu(u)} \, \dfrac{ \mu^n(Eg)  }{ \mu^{n+1}(Eg) }.
\]
From Lemma~\ref{lem:m_nu} we derive that $m_{n+1}^{Eg}(u') \to \mu(u')$ and  $m_n^{Eg}(u) \to \mu(u)$ as $n \to\infty$.
Since $\eps > 0$ was arbitrary, we have $\dfrac{ \mu^{n}(Eg)  }{ \mu^{n+1}(Eg) } \to 1$ as $n \to \infty$.
The same result can be obtained for $2$ using some $u'' \in \alpha_{k+2}$, and
$T^{2}(u'') = u$, with $e \in \psi_{2}(u'')$.
Therefore
$$
\lim_{n\to\infty} \frac{\mu^{n+2}(Eg)}{\mu^{n+1}(Eg)}
=
\lim_{n\to\infty} \frac{\mu^{n+2}(Eg)}{\mu^n(Eg)}
\frac{\mu^n(Eg)}{\mu^{n+1}(Eg)} = 1,
$$
as required.
\end{proof}

\section{Proof of Theorem~\ref{thm:dc}}
\label{sec:dc}

We first obtain  a version of Lemma~\ref{firstnallemma61ds}
for continuous groups. Let $m_n^{Eg}$ as defined in~\eqref{eq:mnc}.

\begin{lemma}\label{firstnallemma61dsCont}Suppose that $G$ is continuous and assume \eqref{eq:Cond-psic}.
For every $\eps > 0$, $k \in \N$ and $x_0 \in X$,
there are $n > \ell > k$ such that if $u$ is the $\ell$-cylinder containing $x_0$, we have
 \[
 \left|m_n^{Eg}(u)-\frac{ \mu^{n-k}( Egg_1^{-1}) }{\mu^n(Eg)} \, \mu(u) \right| < \eps,
 \]
 where $g_1 = \psi_k(x_0)$.
\end{lemma}

\begin{proof}
Let $g \in G$, $k \in \N$ and $x_0 \in X$ be given, and set $g_1 = \psi_k(x_0)$. Let $\eps > 0$ be arbitrary and take $\eps' \in (0,\eps)$ so that $(1+\eps')^5 < 1+\eps$.
Choose $\ell > k$ so large, that $\psi_k$ varies so little on the
$\ell$-cylinder $u$ that Lemma~\ref{lem:hd} applies with $n_0 = \ell$ for $\eps'$.
Choose $n > \ell$ such that $\lambda^{n-\ell} < \eps'$.
Then
\begin{align*}
    m_n^{Eg}(u) =\ &
     \dfrac{1}{\mu^n(Eg)} \sum_{\stackrel{a\in\alpha_n, a \subset u}{\psi_n(a) \cap Eg \neq \emptyset} }\frac{\mu(a \cap \psi_n^{inv}(Eg))}{\mu(a)} \, \mu(a) \\
    \text{by } \eqref{eq:phi-holder} \sim_{1+\eps'} \ & \dfrac{ \mu^{n-k}(Egg_1^{-1}) }{ \mu^n(Eg) } \dfrac{1}{ \mu^{n-k}(Egg_1^{-1}) }  \sum_{\stackrel {a\in\alpha_n, a \subset u}{\psi_n(a) \cap Eg \neq \emptyset} } \frac{\mu(a \cap \psi_n^{inv}(Eg))}{\mu(a)} \, e^{\varphi_k(\xi^k_u)} \mu(a')
    \\
    \sim_{1+\lambda^{n-k}}\ & \dfrac{ \mu^{n-k}(Egg_1^{-1}) }{ \mu^n(Eg) } \dfrac{1}{ \mu^{n-k}(Egg_1^{-1}) }  \sum_{\stackrel{a \subset u, a' = T^k(a) \in \alpha_{n-k}}{\psi_n(a) \cap Eg \neq \emptyset}}
    \!\!\!\! \!\!\!\! \frac{\mu(a' \cap T^k(\psi_n^{inv}(Eg)))}{\mu(a')} \, e^{\varphi_k(\xi^k_u)} \mu(a'),
\end{align*}
where in the last line we used the H\"older distortion control \eqref{eq:phi-holder}, \eqref{eq:IMT} and the fact that $T^k:u \to u' \in \alpha_{n-k}$ is a bijection.

Take now $n > n_2$ as in Lemma~\ref{lem:hd}. Then
\begin{eqnarray*}
 \mu(a' \cap T^k(\psi_n^{inv}(Eg))) &=&  \mu(\{x \in a' : \psi_n(T_u^{-k}x ) \in Eg)\}) \\
 &=&  \mu(\{x \in a' : \psi_{n-k}(x) \in Egg_1^{-1} \,  g_1 \psi_k(T_u^{-k}x )^{-1}\}) \\
 &\sim_{1+\eps'} & \mu(\{x \in a' : \psi_{n-k}(x) \in Egg_1^{-1}\}) \\
 &=& \mu(a' \cap \psi_{n-k}^{inv}(Egg_1^{-1})).
\end{eqnarray*}

Choose $n > n_1 \geq n' > n_0$ according to \eqref{eq:Cond-psic}
where $n_0$ is such that $\lambda^{n_0-k} < \eps'$,
and also that $m^{Eg}_n(u') \sim_{1+\eps'} \mu(u')$ by  Lemma~\ref{lem:m_nu}.
Then, using H\"older continuity of $\varphi$ as in \eqref{eq:phi-holder}, we estimate
\begin{align*}
 \dfrac{1}{ \mu^{n-k}(Egg_1^{-1}) } & \sum_{\stackrel{a' \in \alpha_{n-k}}{\psi_{n-k}(a')=g\psi_k(u)^{-1}}}
 \frac{\mu(a' \cap \psi_{n-k}(Egg_1^{-1}))}{\mu(a')} \, e^{\varphi_k(\xi^k_u)} \, \mu(a') \\
& \sim_{1+\lambda^{n'-k}}\
   \  e^{\varphi_k(\xi^k_u)} \
 m^{Egg_1^{-1}}_{n-k}(u') \\
\text{by Lemma~\ref{lem:m_nu}} & \sim_{1+\eps'}\
   e^{\varphi_k(\xi^k_u)} \ \mu(u') \\
   \text{by } \eqref{eq:IMT} \qquad &=  \mu(u).
\end{align*}
Since  $(1+\eps')^5 < 1+\eps$, we conclude
$m_n^{Eg}(u) \sim_{1+\eps} \dfrac{ \mu^{n-k}(Egg_1^{-1}) }{ \mu^n(Eg) }  \mu(u)$.
As $\varepsilon>0$ was arbitrary, the conclusion follows.
\end{proof}

Using Lemma~\ref{lem:m_nu} together with Lemma~\ref{firstnallemma61ds}
or  Lemma~\ref{firstnallemma61dsCont}, we can complete\\

\begin{proof}[Proof of Theorem~\ref{thm:dc}]
 Item (a). By Lemma~\ref{lem:m_nu} and Lemma~\ref{firstnallemma61ds} with $g=e$ and $g_1=\psi_k(u)^{-1}$,
 $\frac{\mu^{n-k}(g_1)}{\mu^n(e)}\to 1$, as $n\to\infty$. So, $\frac{\mu^{n}(g_1)}{\mu^{n+k}(e)}\to 1$, as $n\to\infty$. Also, by Lemma~\ref{firstnallemma61ds} with $e=g=\psi_k(u)^{-1}$, $\frac{\mu^{n+k}(e)}{\mu^{n}(e)}\to 1$, as $n\to\infty$. Thus,
 \begin{align*}
  \frac{\mu^{n}(g_1)}{\mu^n(e)}=\frac{\mu^{n}(g_1)}{\mu^{n+k}(e)}\frac{\mu^{n+k}(e)}{\mu^{n}(e)}\to 1,
 \end{align*}
as required.

Item (b) follows similarly, with the use of Lemma~\ref{firstnallemma61dsCont}
instead of Lemma~\ref{firstnallemma61ds}.
\end{proof}

\section{Verifying~\eqref{eq:Cond-psi} and~\eqref{eq:Cond-psic} together
with~\eqref{eq:contmix}
for group extensions of Gibbs Markov maps}
\label{sec:verif}

 \subsection{Verifying~\eqref{eq:Cond-psi} for finite groups}

 Recall that $m_G$ is the left Haar measure on the group $G$. If $G$ is finite, then $m_G$ can be normalized
to a probability measure, namely $m_G(g) = \frac{1}{\#G}$ for all $g \in G$.
In this case, \eqref{eq:Cond-psi} is a direct consequence of mixing of the skew-product
$T_\psi(x,g) = (T(x), g \psi(x))$,
namely
\begin{align*}
\lim_{n\to\infty} \mu(b) \mu(\{x \in X: \psi_n(x) = g^{-1}\}) & =
 \lim_{n\to\infty} \int 1_{ b\times G} \cdot 1_{X \times \{ g^{-1} \} } \circ T_\psi^n \, d\hat\mu \\
 & = \hat\mu(b\times G) \hat\mu(X \times \{ g^{-1} \}),
\end{align*}
where the final equality follows by $T_\psi$-invariance of $\hat\mu = \mu \otimes m_G$.

\begin{lemma}
 If $G$ is finite, then the skew-product $T_\psi:X \times G \to X \times G$ is exponentially mixing w.r.t.\ $\mu \otimes m_G$.
\end{lemma}

\begin{proof}
Take the unit element $e \in G$ and consider the first return map $T_\psi^\tau(x,e)$ to $X \times \{ e \}$, where $\tau(x) = \min\{j > 0 : \psi_j(x) = e\}$.
This is a Gibbs-Markov map with distortion constant $C/(1-\lambda)$, where $C$ is the Gibbs (distortion) constant of $T$ and $\lambda \in (0,1)$ is such that
the expansion $\frac{d\mu \circ T_\psi}{d\mu} \geq 1/\lambda$, i.e., the potential $\varphi =  \frac{d\mu}{d\mu\circ T} \leq \log \lambda$.

The return map $T^\tau$ is defined $\mu$-a.e., which follows from the stronger statement that the return time $\tau$ has exponential tails:
There are $\lambda' \in (0,1)$ and $C' > 0$ such that
\begin{equation}\label{eq:exp-tails}
 \mu(\{x \in X : \tau(x) > n\}) \leq C' {\lambda'}^n \qquad \text{ for all } n \in \N.
\end{equation}
Then the mixing of $T_\psi$ follows directly from Young's result \cite[Theorem 2.II(b)]{young}
on the mixing of the Young tower.

It remains to prove \eqref{eq:exp-tails}. Since $N := \#G < \infty$, for every $g \in G$
there is $n < N$ such that $\mu(R_g) > 0$ for $R_g = \{ x \in X : \psi_n(x) = g \}$.
The distortion bound of $T_\psi^m$ (uniform over $m \in \N$) gives that
for all $g \in G$, $m \in \N$ and $a \in \alpha_m$ such that $\psi_m(a) = g^{-1}$,
we have
\[
\frac{\mu(\{x \in a : T^m(x) \in R_g\})}{ \mu(a) } \geq \frac1C \mu(R_g).
\]
Note that if $x \in a$ (so $\psi_m(x) = g^{-1}$) and $T_\psi^m(x) \in R_g$, then $\psi_{m+n}(x) = g \cdot g^{-1} = e$, i.e., $\tau(x) \leq m+N$.

Let $u = \min\{\frac1C \mu(R_g) : g \in G\}$. Then, for all $r \in \N$,
\[
\mu(\{x \in X : \tau(x) >r + N\}) \leq (1-u) \mu(\{x \in X : \tau(x) \geq r\}).
\]
That is, for each $r \in \N$, the set  of points $x$ with $\tau(x) \geq r$ is so that the proportion of points
with $\tau(x) \geq r+N$ is less than $(1-u)$.
Thus,~\eqref{eq:exp-tails} holds with $\lambda' = (1-u)^{1/N}$.
\end{proof}
\begin{remark}
 Regarding compact group extensions we resume to a remark.
 As far as we know, mixing (in fact quantitative mixing) is known
 for semisimple compact group extensions of shifts of finite type.
 The main reference is Dolgopyat~\cite{Dolgopyat02}. While it seems very plausible that this result can be generalized to compact group extensions of GM maps, we could not find a work that does this. 
 Provided mixing for compact groups of GM maps,~\eqref{eq:Cond-psic} is an immediate consequence.
\end{remark}

\subsection{Verifying~\eqref{eq:Cond-psi} and~\eqref{eq:Cond-psic} together
with~\eqref{eq:contmix} for  certain locally compact abelian groups}\label{subsec:ab}

Throughout this section, we assume that the group $G$ is locally compact abelian (LCA), second countable and, further that: 
\begin{itemize}
 \item[(a)] if $G$ is discrete we require that $G$ is finitely
generated, that is $G\cong \Z^d\times F$, for some finite abelian group $F$;
\item[(b)]if $G$ is continuous, we assume that $G$ is compactly generated, that is if there exists a compact subset $K\subset G$ such that $\langle K\rangle=G$.
\end{itemize}
In case (b), we have $G\cong \R^d\times \Z^m\times K$, where $K$ is a compact abelian group, see~\cite[Chapter 2]{HR}.

Some facts established in~\cite{AD01,AD02} will be instrumental for the proof. A summary of assumptions on $\psi$ is provided at the end of this subsection.

We first recall the terminology in~\cite[Chapter VII]{Katzn}
and~\cite{Fol}.
The dual $\hat G$ is the set of characters, that is continuous homomorphisms $\chi:G\to \S^1$, where $\S^1$ is the multiplicative group of complex number of modulus $1$.
Write $\chi(g)$ as $\langle \chi, g\rangle$ or $e^{i\chi g}$, and so
$|\chi(g)|=1$ and $\chi(g+g')=\chi(g)\chi(g')$ (see~\cite[Chapter 4]{Fol},~\cite[Chapter VII]{Katzn}).
Here, we mean $\S^1=\mathbb{T}$, when we speak of the set of characters, and in the notation $e^{i\chi g}$ we also think of $\chi\in \S^1$, where $\S^1=\{z\in \mathbb{C}:|z|=1\}$.

\subsubsection{Transfer operator along with its perturbed version}
\label{trop}
Given $v:X\to\R$, let
\[
D_av=\sup_{x,x'\in a,\,x\neq x'}\frac{ |v(x)-v(x')| }{ d_\beta(x,x')},\qquad |v|_\beta=\sup_{a\in\{a\}}D_a v.
\]
The space $\cB_\beta\subset L^\infty(\mu)$ consisting of the functions
$v:X\to\R$ such that $|v|_\beta<\infty$ with norm
$\|v\|_{\cB_\beta}=\|v\|_{L^\infty(\mu)}+|v|_\beta<\infty$ is a Banach space. It is known that the transfer operator $L: L^1(\mu)\to L^1(\mu)$, $\int_X L v w\, d\mu_Y=\int_Y v w\circ T\, d\mu_Y$ has a spectral gap in $\cB_\beta$ (see,~\cite[Chapter 4]{Aaronson}). In particular, this means that $1$ is a simple eigenvalue, isolated in
the spectrum of $L$.
It is well known (see~\cite[Chapter 4]{Aaronson}) that
for $v\in\cB_\beta$,
\begin{equation}\label{eq:op}
 L^n v=\int_X v\, d\mu+Q^n v,\quad \|Q^n\|_{\cB_\beta}\le C\delta^n,
 \text{ for some } C>0,\delta\in (0,1).
\end{equation}

Define the perturbed transfer operator $L_{\chi} v= L(\chi\circ\psi\, v), v\in L^1(\mu)$ and consider its restriction to $\cB_\beta$:
\[
 L_{\chi} v= L(\chi\circ\psi\, v),\quad \chi\in\hat G, v\in\cB_\beta.
\]
The Fourier transform of $\psi_n$ w.r.t.\  $\mu$ is given by $\mathbb E_\mu(\chi\circ\psi_n\, 1)=\int_X L_{\chi}^n 1\, d\mu$.
It is known~\cite{AD01, AD02} that
\begin{align}\label{psiHcsT}
 \|L_{\chi}-L\|_{\cB_\beta}\to 0, \text{ as } \chi\to 0.
\end{align}
In~\cite{AD01, AD02}, to obtain the previous displayed continuity property of $L_\chi$,
they assume that the cocycle $\psi$ is  H\"older on each $a\in\alpha$ and that $\sup_{a\in\alpha}|D_a\psi|<\infty$. However, the proof goes similarly
if one only assumes that $\psi$ is 
H\"older on each $a$ and $\psi\in L^\gamma(\mu)$ for some $\gamma>0$.

Given
$
D_a(\psi)=\sup_{x,y\in a}\dfrac{d_G(\psi(x),\psi(y))}{d_\beta(x,y)},
$
one says that $\psi$ is 
H\"older on the partition element $a$, 
if $D_a(\psi)<C_a$, for some constant that, possibly, depends on $a$. 
Assuming that $\psi$ is 
H\"older on each $a$ and $\psi\in L^\gamma(\mu)$ for some $\gamma>0$, the proof of~\eqref{psiHcsT} goes similarly that the one in~\cite{AD01, AD02}.
For a precise reference where this was done (with the purpose of obtaining a  precise continuity estimates) see~\cite[Lemma 5.2]{MelTer17} for $G=\R$ and~\cite[Proof of Lemma 6.2]{MTjmd}
for the case $G=\mathbb T^d$. \\

To ensure that the eigenvalues of $L_{\chi}$
do not lie on the unit circle $\mathbb{S}^1=\{z\in\C: |z|=1\}$ (unless in trivial case),
we need an aperiodicity assumption, formulated as in~\cite[Section 3]{AD01}.
\begin{description}[style=multiline,labelwidth=3cm,align=parright, leftmargin=0pt, itemindent=31pt]
 \item[(Ap)]\makeatletter\def\@currentlabel{Ap}\makeatother\label{eq:app}
For each $\chi\in\hat G$, we assume that
 $\psi$ does not satisfy the cohomological equation. That is if
 $
  \chi\circ\psi(x)=\lambda\frac{v(x)}{v\circ T(x)}$ for $\mu$- a.e. $x$,
with $v:X\to \T$, $\lambda\in \T$,
then $v\equiv 1$ and $\lambda=1$.
\end{description}

\medskip

Under \eqref{eq:app}, it is shown in~\cite{AD01}
that for all $\chi$, outside a neighborhood of $0$, $\|L_\chi^n\|_{\cB_\beta}\le \delta^n$, for some $\delta\in(0,1)$.

 In short, as established in~\cite{AD01}, there exists $\eps>0$ so that for 
$\chi\in  B_{\hat G}(0,\eps)$
\[
 L_{\chi} ^n v=\lambda_\chi^n\Pi_\chi+ Q_\chi^n,
\]
where $\{\lambda_\chi\}$ is a family of leading eigenvalues with $\lambda_0=1$, 
$\{\Pi_\chi\}$ is the family of eigenprojection associated with $\{\lambda_\chi\}$,
where $\Pi_0 v=\int_X v\,d\mu$,
and $\|Q_\chi^n\|_{\cB_\beta}\le \delta^n$, for some $\delta\in(0,1)$. The aperiodicity condition ~\eqref{eq:app}  implies that there exists $\delta_0\in (0,1)$ so that
\begin{equation*}
\|L_{\chi} ^n\|_{\cB_\beta}\le \delta_0^n, \text{ for all }|\chi|\ge \eps, \chi\in\text{a compact set of }\hat G.
\end{equation*}
If $\widehat G$ is compact, which is equivalent to  $G$ being discrete (see, for instance,~\cite[Proposition 4.5]{Fol}), then the restriction to a compact set is not needed.

Let $\delta_1=\min\{\delta_0,\delta_1\}$. Using the previous two displayed equations and recalling that $\cB_\beta\subset L^\infty(\mu)$ and that 
$\Pi_0 v=\int  v\, d\mu$,
\begin{align}\label{eq:Lpe}
\nonumber\int_X L_{\chi}^n v\, d\mu &=\lambda_\chi^n\int_X \Pi_{\chi} v\, d\mu+O(\delta_1^n)
=\lambda_\chi^n\int_X \Pi_0 v\, d\mu+\lambda_\chi^n\int_X (\Pi_{\chi}-\Pi_0) v\, d\mu
+O(\delta_1^n)\\
&=\lambda_\chi^n \int  v\, d\mu\left(1+\frac{\int_X (\Pi_{\chi}-\Pi_0) v\, d\mu}{\int  v\,d\mu}\right) +O(\delta_1^n),
\end{align}
 for any $v\in\cB_\beta$ with $\int  v\, d\mu\ne 0$.

\subsubsection{Inversion formula}
Let $d\chi$ be a Haar measure on $\hat G$. 
Via the inversion of the Fourier transform (see, for instance,~\cite[Theorem 4.22]{Fol}, \cite[Chapter VII]{Katzn}), given  a function $h\in L^1(G)$ w.r.t.\  to the Haar measure $m_G$,
\begin{align}\label{eq:invgen}
 h(y)=\int_{\hat G}\chi(-y)\hat{h}(\chi)\, d\chi,\quad \hat{h}(\chi)=\int_G \chi(g) h(g)\, dm_G(g).
\end{align}
If $G$ is discrete (similar to the well known case of $\Z$), take
$h=1_E$ in~\eqref{eq:invgen} for $E=\{\psi_n(x) =g\}$ with $n\in \N$, $x\in X$
and $g\in G$. 
Thus, for any $n\in\mathbb{n}$,  $x\in X$ and $g\in G$, $1_{\{\psi_n(x) -g=0\}}(y)=\int_{\hat G}\chi(-y)\widehat{1_E}(\chi)\, d\chi=\int_{\hat G}\chi(\psi_n(x)-g)\, d\chi$, where in the last equality we have used that $\chi(-y)$ is already captured in $\widehat{1_E}(\chi)$.
Multiplying with $v\in L^1(\mu)$ on both sides, applying $L^n$ on both sides and integrating over the space $X$,
\begin{equation}\label{eq:invab}
 \int_X v1_{\{x\in X:\psi_n(x) =g\}}\, d\mu=\int_{\hat G}\chi(-g)\int_X L_{\chi}^n v\, d\mu\, d\chi.
\end{equation}

As in~\cite[Section 4.3, Chapter 4]{Fol},~\cite[Section 4, Chapter VII]{Katzn}, the case of $G$ continuous but abelian is similar to Fourier analysis on $\R$.
If $G$ is continuous, let $E\subset G$  be an open bounded set  with $m_G(\partial E)=0$. 
Let $\Phi_\delta\in C^\infty (G)$ be  any compactly supported approximate identity and write $f_\delta=\Phi_\delta*1_E$
for the mollification of $1_E$. By a standard mollification argument, $\lim_{\delta\to 0}f_\delta(y)\to 1_E(y)$, $m_G$-a.e. $y\in G$.

Recall~\eqref{eq:invgen} and take $h=f_\delta(\psi_n(x)-g)$ for $n\in\N$, $x\in X$
and $g\in G$.
Let $\hat f_\delta\in\hat G$ be the Fourier transform of $f_\delta$, supported in a bounded set of $\hat G$ and note that
\begin{align}\label{eq:hatf0}
 \hat f_\delta(0)=\int_G  f_\delta(g)\, dm_G(g).
\end{align}
By~\eqref{eq:invgen},
$f_\delta(\psi_n(x)-g)=\int_{\hat G}\chi(-y)\hat f_\delta(\psi_n(x)-g)(\chi)\, d\chi=\int_{\hat G}\chi(\psi_n-g)\hat f_\delta(\chi)\, d\chi$ for any $x\in X$, any $n\in\N$ and any $g\in G$. Here we have used that $\chi(-y)$ is captured in $\hat f_\delta$.
Since $\lim_{\delta\to 0}f_\delta(y)\to 1_E(y)$, $m_G$-a.e. $y\in G$,
we have that for any $x\in X$, any $n\in\N$ and any $g\in G$,
$1_{\{\psi_n(x)\in g+E\}}=\lim_{\delta\to 0}\int_{\hat G}\chi(\psi_n(x)-g)\hat f_\delta(\chi)\, d\chi$.

Applying $v$ on both sides, applying $L^n$ on both sides  and integrating over the space $X$,
\begin{equation}\label{eq:invabcont}
 \int_X v1_{\{x\in X:\psi_n(x)\in g+E\}}\, d\mu=\lim_{\delta\to 0}\int_{\hat G}\chi(-g)\hat f_\delta(\chi)\int_X L_{\chi}^n v\, d\mu\, d\chi.
\end{equation}

Taking $v=1_b$ in~\eqref{eq:invabcont} for $b$ a cylinder and recalling~\eqref{eq:Lpe},
\begin{align}\label{eq:invabb}
 \mu(b&\cap\{x \in X : \psi_n(x) \in E+g\})=\lim_{\delta\to 0}\int_{\hat G}\chi(-g)
 \hat f_\delta(\chi)\int_X L_{\chi}^n 1_b\, d\mu\, d\chi\\
\nonumber &=\mu(b)\lim_{\delta\to 0}\int_{  B_{\hat G}(0,\eps)}\chi(-g)\hat f_\delta(\chi)\lambda_\chi^n \left(1+\frac{\int_X (\Pi_{\chi}-\Pi_0) 1_b\, d\mu}{\mu(b)}\right)\, d\chi+O(\delta_1^n).
\end{align}
A similar formula holds by taking $v=1_b$ in~\eqref{eq:invab}: the only difference is that $\mu(b\cap\{x \in X : \psi_n(x) \in E+g\})$ is replaced by $\mu(b\cap\{x \in X : \psi_n(x) = g\})$ (which is easier to manipulate due to the absence of $\hat f_\delta$).

\subsubsection{Use of (a slightly weaker) version of the symmetry assumption~\eqref{eq:S}}

While the term $O(\delta_1^n)$ in~\eqref{eq:invabb} is dominated by $\int_{B_{\hat G}(0,\eps)}\chi(g)\lambda_\chi^n\, d\xi$, for all $n$ large enough, we need to take a 
 closer look at the term containing $\int_X (\Pi_{\chi}-\Pi_0) 1_b\, d\mu$.
 In this sense, we recall a fact established in~\cite[p.38--39]{AD02}, which relies on the fact that $\lambda_\chi$ is real for all $\chi$.

  The reality of $\lambda_\chi$ is ensured under the assumption that $\psi$ is symmetric in a slightly weaker form of~\eqref{eq:S}, which goes exactly the same without requiring that $S$ is continuous.
  For convenience, we recall this assumption in the form needed in this section: There exists an invertible probability measure preserving transformation $S:X\to X$ so that 
\begin{align}\label{symmetryab}
 S\circ T=T\circ S\text{ and }\psi\circ S=-\psi.
\end{align}

The first part
of~\eqref{symmetryab}, together with the definition of $\varphi$ and the invariance of the measure $\mu$ implies that
$$
\varphi\circ S(y)  = \frac{d\mu}{d\mu \circ T} \circ S (y)= \frac{d\mu \circ S}{d\mu \circ S \circ T} (y) = \frac{d\mu}{d\mu \circ T}(y) = \varphi(y).
$$
Using this together with the second part of~\eqref{symmetryab}, it is shown in~\cite{AD02} that for $v\in\cB$, $L_{-\chi} v(x)= L_{\chi} (v\circ S^{-1})S(x)$.
Taking $v=v_\chi$, where $v_\chi$ is the eigenfunction associated with $\lambda_\chi$,
$L_{-\chi} (v_\chi\circ S)(x)=\lambda_\chi v_\chi\circ S(x)$. 
Recall that $L_\chi$ is continuous. It follows that $L_{-\chi}$ has $\lambda_\chi$
as an eigenvalue.
Since $\lambda_{-\chi}=\overline{\lambda_{\chi}}$, where $\overline{\lambda_{\chi}}$ is the complex conjugate of $\lambda_{\chi}$,
$\lambda_{\chi}$ is real.

Choose $\eta \in (0,\eps)$ small enough so that 
\begin{align}\label{eq:un1}
 u_n(\eta)=\int_{B_{\hat G}(0,\eta)}\lambda_\chi^n\, d\chi>0.
\end{align}

As established
in~\cite[p.39]{AD02},
\begin{equation}\label{eq:etaetapr}
u_n(\eta)=u_n(\eta')(1+o(1)),\text{ as } n\to\infty, \text{ for all }\eta,\eta'\in (0,\eps).
\end{equation}

\subsubsection{The argument for the verification of~\eqref{eq:Cond-psi} and~\eqref{eq:Cond-psic} via a `change of variable'.}

Let $(a_n)$ be any sequence so that $a_n\to\infty$ as $n\to\infty$.
We want to make sense of a change of variables that captures $\chi \to \chi/a_n$.
If $G$ is discrete (case (a)), we assumed that $G\cong\Z^d\times F$, so $\hat G\cong \T^d\times \hat F$, where $\hat F$ is the dual of the finite abelian group.
In this case, the change of coordinates is unproblematic and it goes as if operating on $\Z^d$.

If $G$ is continuous (case (b)) we need to exploit the linear structure of the space $\hat G$. 
Recall that in case (b), we assume that $G$ is compactly generated,
and thus  $G\cong\R^d\times\Z^m\times  K$, where $K$ is a compact abelian group.
 Pontryagin duality, see~\cite[Proposition 4.7]{Fol}, guarantees that $\widehat{G}\cong \widehat{\mathbb{R}^d}\times\widehat{\mathbb{Z}^m}\times \widehat{K}$. 
Note that $\widehat{\mathbb{R}^d}\cong\mathbb{R}^d$ and $\widehat{\mathbb{Z}^m}\cong \T^m$. As already mentioned if a group is compact then its dual is discrete
(again, see~\cite[Proposition 4.5]{Fol}). Thus, $\widehat{K}$ is discrete.
It is known that $\R^{2m}$ is a covering space of $\T^m$ (see, for  instance,~\cite[Theorem 53.3]{MuT}). Thus, the covering space of $\R^d\times\T^m$
is $\R^d\times\R^{2m}$. Since $\hat K$ is discrete, we can define the covering map
\begin{align*}
 p:\R^d\times \R^{2m}\to \R^d\times \T^m\times \{0_{\hat K}\}\text{ so that }p(x)=\chi\text{ and } p(0)=0.
\end{align*}
This is what we are after since we only want to capture a neighborhood of $\widehat{G}$ that contains $0$. We can also ensure that $p$ is a local isometry.
The measure $d\chi$ lifted to the cover is Lebesgue and we write $H:= \R^d\times \R^{2m}$
for the covering space.

Recalling~\eqref{eq:un1}
and using that $p$ is a local isometry (which means we can operate with the same $\eta$),
we obtain that for all $\eta<\eps$,
\begin{align}\label{eq:chvar}
u_n:=u_n(\eta) &= \int_{B_{H}(0,\eta a_n)}\lambda_{p(x)}^n dx
= \frac{1}{a_n^{d+2m}}\int_{B_{H}(0,\eta a_n)}\lambda_{p(x/a_n)}^n
\, dx.
\end{align}
We will use the same `change of variable' inside ~\eqref{eq:invabb}.\\

{\textbf{Step 1: Showing that} \begin{align*}
    \int_{ B_{\hat G}(0,\eps)}\chi(-g)\hat f_\delta(\chi)\lambda_\chi^n \left(1+\frac{\int_X (\Pi_{p(x)}-\Pi_0) 1_b\, d\mu}{\mu(b)}\right) \, d\chi=u_n\int_G f_\delta(g)\, dm_G(g)
\end{align*} \textbf{via the Dominated Convergence Theorem (DCT) and the `change of variables' used in~\eqref{eq:chvar}.}}\\

We will use the  `change of variables' in the sense of~\eqref{eq:chvar} inside the RHS in~\eqref{eq:invabb}.
We omit $\lim_{\delta\to 0}$ and use this at the very end of the argument. Recall that $\chi(-g):=e^{-i\chi g}\to 1$ as $\chi\to 0$.
We have that

  \begin{align*}
   \int_{ B_{\hat G}(0,\eps)}&\chi(-g)\hat f_\delta(\chi)\lambda_\chi^n \left(1+\frac{\int_X (\Pi_{p(x)}-\Pi_0) 1_b\, d\mu}{\mu(b)}\right) \, d\chi \\
   &= \int_{ B_H(0,\eps)} p(x)(-g)\hat f_\delta(p(x))\lambda_{p(x)}^n \left(1+\frac{\int_X (\Pi_{p(x)}-\Pi_0) 1_b\, d\mu}{\mu(b)}\right) \, dx \\
   &=\frac{1}{a_n^{d+2m}} \int_{ B_H(0,a_n\eps)}p(x/a_n)(-g)
   \hat f_\delta(p(x/a_n))\lambda_{p(x/a_n)}^n \left(1+\frac{\int_X (\Pi_{p(x/a_n)}-\Pi_0) 1_b\, d\mu}{\mu(b)}\right)  \, dx.
  \end{align*}
Since $\|\Pi_{\chi}-\Pi_0\|_{\cB_\beta}\to 0$, as $\chi\to 0$,
 $\|\Pi_{p(x/a_n)}-\Pi_0\|_{\cB_\beta} \to 0$ as $n\to\infty$.
Using that $B_\beta\subset L^\infty$,
  $\left|\frac{\int_X (\Pi_{p(x/a_n)}-\Pi_0) 1_b\, d\mu}{\mu(b)}\right|\to 0$, as $n\to\infty$. Also, $\hat f_\delta(p(x/a_n))\to \hat f_\delta(0)$ and
   $p(x/a_n)(-g) \to 1$, pointwise and it is bounded by $1$.
  
   Recall equation~\eqref{eq:etaetapr} and~\eqref{eq:chvar}. Since for all $\eta\in (0,\eps)$, $u_n=u_n(\eta)=\frac{1}{a_n^{d+2m}}\int_{B_{\hat G}(0,\eta a_n)}\lambda_\chi^n\, d\chi$, we can apply the DCT
   to 
   $$p(x/a_n)(g)\hat f(p(x/a_n))\left(1+\frac{\int_X (\Pi_{p(x/a_n)}-\Pi_0) 1_b\, d\mu}{\mu(b)}\right)$$
  and obtain that as $n\to\infty$,
   \begin{align}\label{eq:un}
   &\frac{1}{a_n^{d+2m}}\int_{ B_H(0,a_n\eps) }  p(x/a_n)(-g) \lambda_{p(x/a_n)}^n \left(1+\frac{\int_X (\Pi_{ p(x/a_n) }-\Pi_0) 1_b\, d\mu}{\mu(b)}\right) \, dx \\
 \nonumber &=\hat f_\delta(0)\frac{1}{a_n^{d+2m}}\int_{ B_H(0,a_n\eps)}\lambda_{p(x/a_n)}^n\, dx \left(1+o(1)\right)= u_n\hat f_\delta(0)\left(1+o(1)\right)=u_n(1+o(1))\int_G f_\delta(g)\, dm_G(g),
  \end{align}
where in the last equality we have used~\eqref{eq:chvar}
and~\eqref{eq:hatf0}.\\
  
  {\bf{Step 2: Using DCT in the reverse direction, again together with~\eqref{eq:invabb}.}}\\
  
Using again that $\left|\int_X (\Pi_{p(x/a_n)}-\Pi_0) 1\, d\mu\right|\to 0$,
$|\hat f(p(x/a_n))-\hat f(0)|\to 0$
and $|p(x/a_n)(-g)- 1|\to 0$ as $n \to \infty$, and applying DCT in the reverse direction,
\begin{align*}
&\frac{1}{a_n^{d+2m}}\int_{ B_H(0,a_n\eps)}\lambda_{p(x/a_n)}^n\, dx \, \left(1+o(1)\right)\\
&= \frac{1}{a_n^{d+2m}}\int_{ B_H(0,a_n\eps)} p(x/a_n)(-g)\,\hat f(p(x/a_n))\,\lambda_{p(x/a_n)}^n
\int_X\Pi_{p(x/a_n)} 1\,d\mu\, dx \, \left(1+o(1)\right).
\end{align*}
Recalling~\eqref{eq:invabb}, we get
\begin{align*}
&\mu(b\, \cap\,\{x \in X : \psi_n(x) \in E+g\})\\
&= \lim_{\delta\to 0}\frac{\mu(b)}{a_n^d}
\int_{ B_H(0,\eps a_n)} p(x/a_n)(-g)\,\hat f_\delta(p(x/a_n))\,\lambda_{p(x/a_n)}^n
\int_X \Pi_{p(x/a_n)} 1\,d\mu\, dx\, \left(1+o(1)\right)\\
&= \frac{\mu(b)}{a_n^d}\int_{ B_H(0,\eps a_n) } p(x/a_n)(-g)\,\hat f_\delta(p(x/a_n))\,\int_X L_{p(x/a_n)}^n 1\, d\mu\,dx \, \left(1+o(1)\right)\\
&=\mu(b)\, \int_{B_{\hat G}(0,\eps)} \chi(g)\,\hat f_\delta(\chi)
\int_X L_{\chi}^n 1\, d\mu\,d\chi\left(1+o(1)\right),
\end{align*}
where in the last line we have used the change of variables $\frac{x}{a_n}\to x$, recalled that $p(x)=\chi$ and moved back to the dual group $\hat G$.

By the first equality in~\eqref{eq:invabb} with $v=1_b$ replaced by $v=1_X$,
\begin{align}\label{eq:compE}
 \mu(\{x \in X : \psi_n(x) \in E+g\})=\lim_{\delta\to 0}\int_{B_{\hat G}(0,\eps)} \chi(g)\,\hat f_\delta(g)
\int_X L_{\chi}^n 1\, d\mu\,d\chi.
\end{align}

Assumption~\eqref{eq:Cond-psic} for compactly generated groups follows from the previous two displayed equations.
The verification of~\eqref{eq:Cond-psi} for finitely generated groups goes similarly using~\eqref{eq:invabb}; in this case, the argument is simplified by the absence of $\hat f_\delta$.

\subsection{Proof of Proposition~\ref{prop:comp}}

It follows from~\eqref{eq:compE},~\eqref{eq:invabb} and~\eqref{eq:un}
that
\begin{align}\label{eq:unE}
 \mu(\{x \in X : \psi_n(x) \in E+g\})= u_n(1+o(1)) \lim_{\delta\to 0}\int_G f_\delta(g)\, dm_G(g) = u_n(1+o(1)) m_G(E).
\end{align}
Working with any other open set $A$ in $G$ with $m_G(\partial A)=0$,

\begin{align*}
 \mu(\{x \in X : \psi_n(x) \in A+g\})= u_n(1+o(1)) \lim_{\delta\to 0}\int_G h_\delta(g)\, dm_G(g)= u_n(1+o(1)) m_G(A),
\end{align*}
where $\lim_{\delta\to 0}h_\delta(y)= 1_A(y)$. The conclusion of Proposition~\ref{prop:comp} follows from the the previous two displayed equations.

\subsubsection{Verifying~\eqref{eq:contmix} when $G$ is continuous}

Let $A, F, a$ be as in~\eqref{eq:contmix}. Let $g\in F$ and recall that
$T_\psi^n(x,g)=(T^n x, \psi_n(x)+g)$.
Note that
\begin{align*}\mu\otimes m_G \left((a \times F)\cap T_\psi^{-n}(X\times A)\right)
 &=\mu\otimes m_G \left(\{(x,g)\in a\times F: \psi_n(x) \in A-g\}\right)\\
 &=\mu\otimes m_G \left((a\times F)\cap \{(x,g) \in X\times G: \psi_n(x) \in A-g\}\right).
\end{align*}
Next, we use~\eqref{eq:invabcont} with $v$ replaced by $1_{a\times F}$,
this time integrating w.r.t. $\mu\otimes m_G$. Thus,
\begin{align*}
 \mu\otimes m_G \left((a \times F)\cap T_\psi^{-n}(X\times A)\right)
 &=\lim_{\delta\to 0}\int_{\hat G}\chi(-g)\hat f_\delta(\chi)\int L_\chi^n 1_a 1_F\, d\mu\, d_{m_G}\, d\chi\\
 &=m_G(F)\lim_{\delta\to 0}\int_{\hat G}\chi(-g)\hat f_\delta(\chi)\int L_\chi^n 1_a \, d\mu\,\, d\chi,
\end{align*}
where we have used that $1_{\{ \psi_n(x)\in A+g\} }=\lim_{\delta\to 0}\int_{\hat G}\chi(\psi_n-g)\hat f_\delta(\chi)\, d\chi$. (In particular, the set $E$ in~\eqref{eq:invabcont} is replaced by the set $A$.)

The rest goes word by word as the first part of the verification of~\eqref{eq:Cond-psic} in the previous section. In short we obtain
\begin{align*}
 &\int_{\hat G}\chi(g)\hat f_\delta(\chi)\int L_\chi^n 1_a 1_F\, d\mu\, dm_G\, d\chi\\
 &=\mu(a)m_G(F)\hat f(0)\frac{1}{a_n^{d+2m}}\int_{ B_H(0,a_n\eps)}\lambda_{p(x/a_n)}^n\, dx \left(1+o(1)\right)\\
 &=\mu(a)m_G(F)\hat f_\delta(0)u_n (1+o(1)),
\end{align*}
where in the last equalities we have used~\eqref{eq:un}. To conclude, recall that $\hat f_\delta(0)=\int_G f_\delta(g)\, dm_G(g)$ and take the limit $\delta\to0$.
The conclusion follows after recalling~\eqref{eq:unE}.

\subsubsection{Summary of the assumptions on $\psi$}

\begin{itemize}
 \item[(i)] If $G$ is compactly generated, $\psi$ is H\"older on each $a\in\alpha$. If  $G$ is finitely generated, we assume that $\psi$ is constant on each $a\in\alpha$ .
 \item[(ii)] To handle the $\beta$ seminorm of $\psi$, if $G$ is compactly generated, we further assume that $\psi\in L^\gamma(\mu)$ for some $\eta>0$ (see subsection~\ref{trop}). If $G$ is discrete and $\psi$ is constant on partition elements, 
 $|\psi|_\beta=0$.
  \item[(iii)] $\psi$ is aperiodic as in~\eqref{eq:app}.
 \item[(iv)] $\psi$ is symmetric in the sense of~\eqref{symmetryab}.
\end{itemize}

\section{Pressure function for finitely generated group extensions of GM maps}\label{sec:ap}  

For a countable group $G$,
Theorem \ref{thm:amen} says that $\lim_{n\to\infty}\frac{\mu(\{x\in X:\psi_{n+1}(x)=g\})}{\mu(\{x\in X:\psi_n(x)=g\})}=\lim_{n\to\infty}\frac{\mu^{n+1}(e)}{\mu^n(e)}=1$ when \eqref{eq:Cond-psi} holds.
It is known that $\lim_{n\to\infty}\frac{\mu^{n+1}(e)}{\mu^n(e)}=1$ implies that 
\begin{align}\label{eq:limit1}
    \lim_{n\to\infty} (\mu^n(e))^{1/n}=1.
\end{align}

Establishing~\eqref{eq:limit1} via a ratio limit theorem, it is rather a  difficult path.
Via a different strategy,  \cite{Ma13} and \cite{Ja15} established~\eqref{eq:limit1}, which allows for relating the pressure function for $T$ with the pressure function of $T_\psi$.

The  results in \cite{Ma13} and \cite{Ja15} are based on a symmetry assumption for the Gibbs Markov map $T:X\to X$, the potential $\varphi:X\to \mathbb{R}$ and  the cocycle $\psi:X\to G$.
Roughly speaking, 
this symmetry condition on $T$ and $\psi$ (compare with \eqref{symmetryab}) states that there exists a map $S:\alpha \to \alpha$ such that $S^2=\mbox{id}_{\alpha}$ and $\psi(S(x))=\psi(x)^{-1}$, $x\in X$, in such a way it allows to obtain a Gibbs condition relating the elements in some cylinder $a\in\alpha$ with the elements in its involution $S(a)\in\alpha$.
With these assumptions, \cite[Theorem 4.1]{Ma13} relates the thermodynamic formalism of a symmetric Gibbs Markov map and its extension by $\psi$, with the help of a sequence of i.i.d. random walks over the group $G$.

Inspired by this,  we present an expression for $\lim_{n\to\infty}(\mu^n(e))^{1/n}$ in a general setting, without the assumption of a symmetry condition for either $T$ or $\psi$. 
For this, we define measures $m_n$ over the group $G$, in a similar way to \cite{Ma13}, such that $(G,m_n)$ gives an i.i.d. random walk for each $n\in \mathbb{N}$.
In \cite{Ma13}, the symmetry assumption for $T$ and $\psi$  allowed to deduce that the measures $m_n$ defined in \cite{Ma13} are symmetric, i.e., $m_n(g)=m_n(g^{-1})$ for every $g\in G$.
The symmetry on these measures is crucial to apply Kesten's criteria (see \cite{Ke59}), which is a key part of the proof of \cite[Theorem 4.1]{Ma13}.
Instead of using Kesten's criteria, we use some of the results presented in \cite{DoSh24}, which allow to handle the general case without assuming symmetry for $T$ or $\psi$.

In this section, given a probability measure $\nu$ over $G$, $\nu^{*k}$ denotes the $k$-th convolution of the measure $\nu$ and $\lambda(G,\nu)$ the spectral radius of $\nu$, i.e., $\lambda(G,\nu)=\lim_{k\to\infty}(\nu^{*k}(e))^{1/k}$.
For a Gibbs-Markov map $(X,T,\mu,\alpha)$, with a potential $\varphi:X\to \mathbb{R}$, recall that the \emph{Gurevi\v{c} pressure} is given by 
\begin{align*}
    P_\varphi(T)=\limsup_{n\to\infty}\frac{1}{n}\log Z_a^n, \mbox{ where } Z_a^n:=\sum_{\substack{T^n(x)=x\\x\in a}}e^{\varphi_n(x)},
\end{align*}
for some $a\in \alpha$. 
If $(X,T,\mu,\alpha)$ is transitive and $\varphi$ is a locally H\"older continuous potential, then $P_\varphi(T)$ does not depend on $a\in\alpha$, and the $\limsup$ can be replaced by $\lim$ (see, for instance, \cite{SaLN}).

Next, we present some facts related to the abelianization of a finitely generated group $G$ presented in the works \cite{DoSh21} and \cite{DoSh24}. 
Let $G$ be a finitely generated group.
The group $G^{\rm{ab}}=G/[G,G]$ is an abelian finitely generated group, where $[G,G]$ denotes the commutator subgroup of $G$, that is, $[G,G]$ is the subgroup of $G$ generated by the elements of the form $[a,b]=aba^{-1}b^{-1}$, for $a,b\in G$.
Therefore, the \emph{free torsion abelianization of $G$} is given by $\overline{G}:=G^{\rm{ab}}/G^{\rm{ab}}_T\cong \mathbb{Z}^k$, for some $k\geq 0$, where $G^{\rm{ab}}_T$ denotes the torsion part of  $G^{\rm{ab}}$.
Denote by  $\pi:G\to\overline{G}$ the canonical homomorphism between $G$ and $\overline{G}$.
We refer to ${T}_\psi^{\rm{ab}}:X\times \overline{G}\to X\times \overline{G}$ as the extension of $X$ by the group $\overline{G}$ with cocycle $\overline{\psi}=\pi\circ\psi$, i.e., ${T}_\psi^{\rm{ab}}(x,m)=(T,m+\overline{\psi}(x)), x\in X, m\in\overline{G}.$
 The Gurevi\v{c} pressures of $T_\psi$ and $T_\psi^{\rm ab}$ are defined in a similar way as before,
 \begin{align*}
      P_\varphi(T_\psi)=\limsup_{n\to\infty}\frac{1}{n}\log Z_{a,g}^n,\mbox{ where }Z^n_{a,g}:=\sum_{\substack{T^n(x)=x,x\in a\\
     \psi_n(x)=g}} e^{\varphi_n(x)}, \mbox{ for }a\in\alpha, g\in G
 \end{align*}
and
\begin{align*}
    P_\varphi(T_\psi^{\rm ab})=\limsup_{n\to\infty}\frac{1}{n}\log Z_{a,m}^n,\mbox{ where }Z^n_{a,m}:=\sum_{\substack{T^n(x)=x,x\in a\\
     \psi_n(x)=m}}e^{\varphi_n(x)}, \mbox{ for }a\in\alpha, m\in \overline{G}.
\end{align*}
It is clear that $Z^n_{a,g}\leq Z^n_{a,\pi(g)} $, which implies that $P_\varphi(T_\psi)\leq P_\varphi(T_{\psi}^{\rm ab})$.

The main result of this section reads as follows:

\begin{proposition}\label{Prop:T-and-ext}
    Assume that $T_\psi$ is a topologically transitive extension of $(X,T,\alpha)$.
    If $G$ is an amenable finitely generated group, then 
    \begin{align}\label{eq: equal-abelian}
        P_\varphi(T_\psi)=        P_\varphi(T_\psi^{\rm ab}).
\end{align}
Moreover, there exists a subsequence $(m_{n_i})\subseteq (m_n)$ such that
$$P_\varphi(T_\psi)=\lim_{i\to\infty}\frac{1}{n_i}\log\lambda(G,m_{n_i})+P_\varphi(T).$$
\end{proposition}

\begin{remark}
{\rm 
Theorem 5.1 in \cite{DoSh21} is phrased as: 
 if $G$ is an amenable group, then $P_\varphi(T_\psi^{\rm{ab}})=P_\varphi(T_\psi)$, where $(X,T,\alpha)$ is a subshift of finite type contained in a full-shift over a finite alphabet.
 Proposition \ref{Prop:T-and-ext} is a generalization  for Gibbs Markov maps of \cite[Theorem 4.1]{DoSh21} when the group $G$ is assumed to have the additional condition of being finitely generated.} 
 \end{remark}
The proof of Proposition \ref{Prop:T-and-ext} is presented at the end of this section. 

\subsection{Almost superadditive sequences}
The following lemma is folklore. 
However, we add a proof for completeness of this document.
\begin{lemma}\label{lemma:Fekete}
    Let $(a_n)\subseteq \mathbb{R}$ and $C > 0$ be so that 
    \begin{align}\label{eq:Feketeeq}
        a_{n+m}\geq a_n+a_m+\log C, \mbox{ for all }n,m\in\mathbb{N}.
    \end{align}
    Then, the limit 
        $\lim_{n\to\infty}\frac{a_n}{n}$ exists (although it may be infinite).
\end{lemma}
\begin{proof}
We proceed via (the proof of) Fekete's lemma, but we need to adjust this slightly due to presence of $C > 0$. 
 The rest of the proof is standard, repeating the steps in Fekete's lemma.
 Write $n=mp+q$, with $q<p$, and note that
 \begin{align*}
  \frac{a_n}{n}\ge \frac{a_{mp+q}}{mp+q}\ge \frac{m a_p +a_q-(m+1)\log C^*}{mp+q}
  \ge \frac{m}{m+1}\frac{a_p}{p} +\frac{a_q}{(m+1)p}+\frac{\log C}{p}.
 \end{align*}
 Let $\eps>0$ be arbitrary. 
If $\limsup_r \frac{a_r}{r} < \infty$, then
choose $p$ so that
$\frac{a_p}{p}\ge \lim\sup_r\frac{a_r}{r}-\eps$ and $\frac{\log C}{p} > -\eps$. 
Thus,
\[
 \frac{a_n}{n}\ge \frac{m}{m+1}\lim\sup_r\frac{a_r}{r}+\frac{a_q}{(m+1)p}-2\eps.
\]
Let $m\to\infty$. 
Then $\liminf_{n\to\infty}\frac{a_n}{n}\geq  \lim\sup_{n\to\infty}\frac{a_n}{n}$, and the conclusion follows.

Now, if  $\limsup_r \frac{a_r}{r} = \infty$, then choose $p$ such that
$\frac{a_p}{p}\ge  1/\eps$ and $\frac{\log C}{p} > -\eps$.
Therefore,
\[
 \frac{a_n}{n}\ge \frac{m}{m+1} \frac{1}{\eps} + \frac{a_q}{(m+1)p}-\eps.
\]
Let $m\to\infty$. Then  $\liminf_{n\to\infty}\frac{a_n}{n}\geq \frac{1}{\eps} - \eps$, and since $\eps > 0$ was arbitrary,
$\liminf_{n\to\infty}\frac{a_n}{n} = \infty$.
\end{proof}
We prove that the sequence $(\mu^n(e))$ satisfies \eqref{eq:Feketeeq} for some $C>0$.
\begin{lemma}\label{lemma:subadditive mu}
    There exists $D>0$ such that for each $n,m\in\mathbb{N}$,
    \begin{align*}
        \mu^{n+m}(e)\geq D\mu^n(e)\mu^m(e).
    \end{align*}
\end{lemma}
\begin{proof}
     Let $c_n,a_m$ be an $n$-cylinder and an $m$-cylinder, respectively. 
    Note that
    \begin{align*}
        \mu(c_n\cap T^{-n}a_m)\sim_{D'}\mu(c_n)\mu(a_m),
    \end{align*}
    for some $D'>0$.
    Indeed, the Gibbs property \eqref{Eq: Gibbs prop} guarantees  the existence of $C>0$, such that $\mu(c_n)\sim_{C}e^{\varphi_n(x)}$, $x\in c_n$, and $\mu(a_m)\sim_C e^{\varphi_m(y)}$, $y\in a_m$. 
    Given the $n+m$-cylinder $b_{m+n}=c_n\cap T^{-n}(a_m)$, we obtain that
    \begin{align*}
         \mu(c_n\cap T^{-n}a_m)\sim_C e^{\varphi_{n+m}(x)}=e^{\varphi_n(x)}e^{\varphi_m(T^n(x))}\sim_{C^2}\mu(c_n)\mu(a_m).
    \end{align*}
    From this, we deduce that
    \begin{align}\label{eq:mu^n}
        \mu^{n+m}(e)\sim_{D'}\sum_{\substack{b_{n+m}=c_n\cap T^{-n}a_m\in \alpha_{n+m}\\\psi_{n+m}(b_{n+m})=e}}\mu(c_n)\mu(a_m).
    \end{align}
    Now, using that $\psi$ is constant on cylinders, we conclude that
    \begin{align*}
        \mu^{n+m}(e)=\mu(\psi_{n+m}^{inv}(e))=\sum_{\substack{b_{n+m}\in \alpha_{n+m}\\\psi_{n+m}(b_{n+m})=e}}\mu(b_{n+m}).
    \end{align*}
    If $\psi_n(x)=e$ and $\psi_m(T^nx)=e$, then $\psi_{n+m}(x)=e$. 
    Thus,
    \begin{align*}
        \sum_{\substack{b_{n+m}=c_n\cap T^{-n}a_m\in \alpha_{n+m}\\\psi_{n+m}(b_{n+m})=e}}\mu(c_n)\mu(a_m)&\geq \sum_{\substack{c_{n}\in \alpha_{n}\\\psi_{n}(c_{n})=e}}\sum_{\substack{a_{m}\in \alpha_{m}\\\psi_{m}(a_{m})=e}}\mu(c_n)\mu(a_m)\\
        &=\sum_{\substack{c_{n}\in \alpha_{n}\\\psi_{n}(c_{n})=e}}\mu(c_n)\sum_{\substack{a_{m}\in \alpha_{m}\\\psi_{m}(a_{m})=e}}\mu(a_m)\\
        &=\mu^n(e)\mu^m(e).
    \end{align*}
    This, together with \eqref{eq:mu^n}, gives the conclusion.
\end{proof}
Note that under the assumption that $\psi$ is constant on partition elements and $T$ is topologically mixing, the first part of the short argument provided in~\cite[Proof of Theorem 5.5]{Ma13} remains the same, ensuring that
\[
 P_\varphi(T_\psi)=\limsup_{n\to\infty}\log\left((\mu^n(e))^{1/n}\right).
\]
This part of the argument does not require any symmetry condition of $\mu$, equivalently of $\varphi$.
This implies that
\begin{align*}
\rho:=\limsup_{n\to\infty}(\mu^n(e))^{1/n}=e^{P_\varphi(T_\psi)}.
\end{align*} 
From now on, we assume, without loss of generality, that $P_\varphi(T)=0$, that is, $e^{P_\varphi(T)}=1$. With this assumption, as shown in~\cite[Proposition 5.2]{Ma13} we have that $\rho\le 1$,
with equality only if $G$ is  amenable and $T$, $\psi$, and $\varphi$ are symmetric in the sense of~\cite[Theorem 4.1]{Ma13}.

The following corollary gives the existence and finiteness of the limit we are studying.
\begin{corollary}\label{corollary:existence limit}
    If $e\in \psi(X)$, then the limit
    \begin{align*}
    \lim_{n\to\infty}\frac{1}{n}\log(\mu^n(e))
    \end{align*}
    exists and is finite.
\end{corollary}
\begin{proof}
    If $e\in \psi(X)$, then $\mu^n(e)>0$ by Lemma \ref{lemma: positivemun}. 
The corollary follows directly as an application of Lemma \ref{lemma:subadditive mu} and Lemma \ref{lemma:Fekete} by taking $a_n=\log(\mu^n(e))$, for $n\in\mathbb{N}$.
This finiteness of the limit is a direct consequence of $\rho\leq 1$.
\end{proof}

Corollary \ref{corollary:existence limit} implies the following proposition.
\begin{proposition}\label{prop:limnonab} Suppose that $T$ is topologically mixing and that $e\in\psi(X)$.
Then
\begin{align*}
\rho=    \lim_{n\to\infty}(\mu^n(e))^{1/n}=e^{P_G(T_\psi)}.
\end{align*} 
\end{proposition}
\subsection{Relation between $P_\varphi(T_\psi)$ and $P_\varphi({T})$ }
From now on, assume that $T$ is topologically mixing and that $e\in\psi(X)$.

 Let $a\in \alpha_m$, for some $m\in\mathbb{N}$ and  $\xi\in a$ a fixed element.
For each $n\geq m$, denote 
$$\alpha_n(a)=\{v\in\alpha_n: v\cap a=v \bmod{\mu},\mu(v\cap T^{-n}(a))>0\}.$$
For every $n\in\mathbb{N}$, define the operator $P_n: \ell^2(G)\to\ell^2(G)$ by
\begin{align*}
    P_n(f)(h):=\sum_{\substack{v\in \alpha_n(a)}}e^{\varphi_n(T_v^{-1} \xi)}f(h\psi(v)^{-1}), \; f\in \ell^2(G) \mbox{ and }h\in G. 
\end{align*}
Here, $T_v^{-1}: T^n(v)\to v$ is the inverse map of $T^n:v\to T^n(v)$, $v\in\alpha_n$.
Note that $P_n(1)\equiv L^n(1_{a})(\xi)$, which is finite if $P_\varphi(T)<\infty$.
Moreover, \cite[Proposition 3.2]{SaLN} guarantees that $P_\varphi(T)=\lim_{n\to\infty}\frac{1}{n}\log P_n(1)$.

For every $n\in\mathbb{N}$, define the measure $m_n$ associated to $P_n$ as follows:
\begin{align*}
    m_n(g):=\dfrac{1}{P_n(1)}\sum_{\substack{v\in \alpha_n(a)\\\psi_n(v)=g}}e^{\varphi_n(T_v^{-1}(\xi))},\; g\in G.
\end{align*}

It is not necessarily true that every measure $m_n$ is non-degenerate.
However, we have the following.
\begin{lemma}\label{Lemma:subseq}
There exists $s\in\mathbb{N}$ such that $m_{is}$ is a non-degenerate measure over $G$, for each $i\in\mathbb{N}$.    
\end{lemma}
\begin{proof}
    Since $G$ is finitely generated, there exist $g_1,\dots, g_n\in G$ that generate the group $G$ as a semigroup. 
    The transitivity of $T_\psi$ guarantees that for each $i\in\{1,\dots, n\}$, there exist $s_i\in\mathbb{N}$ and $v\in \alpha_{s_i}(a)$ such that $\psi_{s_i}(v)=g_i$.
    In other words, $m_{s_i}(g_i)> 0$.
    Taking $s$ as the least common multiple of $s_1,\dots, s_n$, we deduce that $m_s(g_i)>0$ for every $i\in\{1,\dots,n\}$.
    Consequently, the measures $m_{is}$, $i\in\mathbb{N}$, are non-degenerate.
\end{proof}

For every $i\in\mathbb{N}$, let $n_i=is$ with $s\in\mathbb{N}$ the number given by the previous lemma. 
Then, \cite[Theorem 1.1]{DoSh24} implies that
\begin{align}\label{eq:spectrumDS}
\lambda(G,m_{n_i})=\lambda(\overline{G},\overline{m_{n_i}})=\phi_{\overline{m_{n_i}}}(x_{n_i}),
\end{align} where $\overline{m_{n_i}}$ is the push-forward measure of $m_{n_i}$ under the map $\pi:G\to\overline{G}$, i.e., $\overline{m_{n_i}}=\pi_*(m_{n_i})$, the map $\phi_{\overline{m_{n_i}}}:\mathbb{R}^k\to \mathbb{R}$ is given by
\begin{align*}\phi_{\overline{m_{n_i}}}(x)=\sum_{m\in\mathbb{Z}^k}e^{\langle m,x\rangle }\overline{m_{n_i}}(m), x\in\mathbb{R}^k,
\end{align*}
and $x_{i}\in\mathbb{R}^k$ is the unique element such that $\phi_{\overline{m_{n_i}}}(x_i)=\inf_{x\in\mathbb{R}^k}\phi_{\overline{m_{n_i}}}(x)$.
To ease notation, we denote $\phi_i=\phi_{\overline{m_{n_i}}}$.

Next, we need to argue that $\lim_{i\to\infty}\frac{1}{n_i}\log \phi_i(x_i)$ 
exits, as in Corollary~\ref{cor:69} below.
To do so, we define a new map.
For every $i\in \mathbb{N}$, let $\tilde{\phi_i}(x):\mathbb{R}^k\to\mathbb{R}$ be a map defined by
   \begin{align}\label{eq: def phitilde}
       \tilde{\phi_i}(x):=\sum_{m\in\mathbb{Z}^k}e^{\langle m,x\rangle}\sum_{\substack{v\in\alpha_{n_i}(a)\\\psi_{n_i}(v)=g}}e^{\varphi_{n_i}(T^{-1}_v(\xi))}=P_{n_i}(1)\phi_i(x), x\in\mathbb{R}^k.
   \end{align}
   Note that the sequence $(\log\tilde{\phi}_i(x_{i}))\subseteq \mathbb{R}$ satisfies \eqref{eq:Feketeeq}, as expressed in the following lemma.
\begin{lemma}\label{lemma:subadditivephi}
  There exists $D>0$ such that
  \begin{align*}
      \tilde{\phi}_{i+j}(x_{i+j})\geq D\tilde{\phi_i}(x_{i})\tilde{\phi_j}(x_{j}),\;\mbox{ for all }i,j  \in\mathbb{N}.
  \end{align*}
\end{lemma}
\begin{proof}
    Applying the Gibbs property \eqref{Eq: Gibbs prop}, we deduce that for every $i,j\in\mathbb{N}$, and $v\in \alpha_{n_i}(a)$, $w\in \alpha_{n_j}(a)$, it holds 
    $e^{\varphi_{n_i}(T_v^{-1}x)}\leq C^2 e^{\varphi_{n_i}(T_v^{-1}T_w^{-1}\xi)}$  and $e^{\varphi_{n_j}(T_w^{-1}x)}\leq C^2 e^{\varphi_{n_j}(T_w^{-1}\xi)}$, for each $x\in X$.
    Let $i,j\in \mathbb{N}$ and $x\in\mathbb{R}^k$. 
    We have that,
    \begin{align*}
        \tilde{\phi_i}(x)\tilde{\phi_j}(x)
\leq \;&C^2\sum_{m,t\in\mathbb{Z}^k} e^{\langle m+t, x\rangle} \sum_{\substack{v\in \alpha_{n_i}(a)\\\psi_{n_i}(v)=m}}\sum_{\substack{w\in \alpha_{n_j}(a)\\\psi_{n_j}(w)=t}}e^{\varphi_{n_i}(T^{-1}_vT^{-1}_w(\xi))}e^{\varphi_{n_j}(T^{-1}_w\xi)}\\
        \leq \;&C^2\sum_{m,t\in\mathbb{Z}^k} e^{\langle m+t, x\rangle} \sum_{\substack{v\in \alpha_{n_i}(a)\\\psi_{n_i}(v)=m}}\sum_{\substack{w\in \alpha_{n_j}(a)\\\psi_{n_j}(w)=t}}e^{\varphi_{n_i+n_j}(T^{-1}_vT^{-1}_w(\xi))}\\
        \leq \;&C^2\sum_{m,t\in\mathbb{Z}^k} e^{\langle m+t, x\rangle} \sum_{\substack{v\in \alpha_{n_i+n_j}(a)\\\psi_{n_i+n_j}(v)=m+t}}e^{\varphi_{n_i+n_j}(T^{-1}_v(\xi))}\\
        =&\; C^2 \widetilde{\phi}_{i+j}(x),
        \end{align*}
        From the previous expression, \eqref{eq: def phitilde} and the definition of the $x_{i}$'s, we obtain that $\tilde{\phi_i}(x_i)\tilde{\phi_j}(x_j)\leq C^2 \tilde{\phi}_{i+j}(x_{i+j})$.
        This concludes the proof.
\end{proof}

Using Lemmas \ref{lemma:Fekete},~\ref{Lemma:subseq} and \ref{lemma:subadditivephi} we obtain 

\begin{corollary}\label{cor:69}
The limit $\lim_{i\to\infty}\frac{1}{n_i}\log \phi_i(x_i)$ exists and is finite. 
\end{corollary}
\begin{proof}
Let $s\in \mathbb{N}$ be the number given by Lemma \ref{Lemma:subseq} and $n_i=is$, for every $i\in\mathbb{N}$.
    From Lemma \ref{lemma:subadditivephi} and Lemma \ref{lemma:Fekete}, applied to the sequence $a_i=\log \tilde{\phi_i}(x_i)$, we obtain  that $\lim_{i\to\infty}\frac{1}{n_i}\log\tilde{\phi_i}(x_i)=\frac{1}{s}\lim_{i\to\infty}\frac{1}{i}\log\tilde{\phi_i}(x_i)$ exists.
    As $\tilde{\phi_i}(x_i)=P_{n_i}(1)\phi_i(x_i)$, we have that
\begin{align*}
    \frac{1}{n_i}\log\tilde{\phi_i}(x_i)=\frac{1}{n_i}\log P_{n_i}(1)+\frac{1}{n_i}\phi_i(x_i).
\end{align*}
Since $\lim_{i\to\infty}\frac{1}{n_i}\log P_{n_i}(1)=P_G(T)$ exists, the proof concludes.
\end{proof}

Using   Lemma~\ref{Lemma:subseq} and Corollary~\ref{cor:69}, we complete
\begin{proof}[Proof of Proposition \ref{Prop:T-and-ext}]
    The transitivity of the extension $T_\psi$ implies the existence of an element $v\in \alpha_l(a)$, for some $l\in\mathbb{N}$, so that $\psi_l(v)=e$.
    For every $w_1,\dots, w_k\in \alpha_n(a)$, it holds that $w^*=\bigcap_{j=1}^k T^{-(j-1)(n+l)}(w_j)\cap\bigcap_{j=1}^kT^{-(j-1)(n+l)-n}(v)$ belongs to $\alpha_{k(n+l)}(a)$.
    Notice that using twice the Gibbs property \eqref{Eq: Gibbs prop}, we have that $C^{-2}e^{\varphi_s(z)}\leq e^{\varphi_s(y)}$ for each $z,y\in v$, $v\in\alpha_s$ and $s\in \mathbb{N}$.
    For each $x\in w^*$, note that $T^{(j-1)(n+l)}(x)\in w_j$ and $T^{(j-1)(n+l)+n}(x)\in v$,  for all $j\in \{1,\ldots, k\}$. 
    Therefore,
    \begin{align*}
      e^{\varphi_{k(n+l)}(x)}&=\prod_{j=1}^{k}e^{\varphi_{j}(T^{(j-1)(n+l)}(x))}\prod_{j=1}^{k}e^{\varphi_{j}(T^{(j-1)(n+l)+n}(x))}\\
      &\geq \left[C^{-2k}\prod_{i=1}^k e^{\varphi_n(T^{-1}_{w_i}\xi)}\right]\left[C^{-2} e^{\varphi_l(T^{-1}_{v}\xi)}\right]^k.
    \end{align*}
    On the other hand, for every $k,n\in\mathbb{N}$, it holds
    \begin{align}
        \nonumber\langle \bm{1}_e,P_n^k(\bm{1}_e)\rangle=&\sum_{\substack{w_1,\dots, w_k\in \alpha_n(a)\\ \psi(w_1)\psi(w_2)\cdots\psi(w_k)=e}}e^{\varphi_n(T^{-1}_{w_1}(\xi))+\dots+\varphi_n(T^{-1}_{w_k}(\xi))}\\
        \label{eq:inner}=& \ {(P_n(1))^k}m_n^{*k}(e).
    \end{align}
    Hence,
    \begin{align*}
      \left[C^{-4} e^{\varphi_l(T^{-1}_{v}\xi)}\right]^k  \langle \bm{1}_e,P_n^k(\bm{1}_e)\rangle\leq \sum_{\substack{T^{k(n+l)}(x,e)=(x,e)\\ x\in a}}e^{\varphi_{k(n+l)}(x)}=Z_{a,e}^{k(n+l)}.
    \end{align*}
    Consequently,
    \begin{align*}
      \log\left[C^{-4} e^{\varphi_l(T^{-1}_{v}\xi)}\right]+\log(P_n(1))+\limsup_{k\to\infty}\frac{1}{k} \log \big\langle \bm{1}_e,\left(\dfrac{P_n}{P_n(1)}\right)^k(\bm{1}_e)\big\rangle\leq \limsup_{k\to\infty}\frac{1}{k}\log Z_{a,e}^{k(n+l)}.
    \end{align*}
     Equation \eqref{eq:inner} implies
     \begin{align*}
      \limsup_{k\to\infty}\frac{1}{k}  \log\big\langle \bm{1}_e,\left(\dfrac{P_n}{P_n(1)}\right)^k(\bm{1}_e)\big\rangle=\log\lambda(G,m_n).   
     \end{align*}
     Therefore, for the subsequence $(n_i)$ given by $n_i=is$, with $s\in \mathbb{N}$ the element given by Lemma \ref{Lemma:subseq}, we obtain that
      \begin{align}\label{eq:1-radius}
      P_\varphi(T)+\lim_{i\to\infty}\frac{1}{n_i}\log \lambda(G,m_{n_i})\leq \lim_{i\to\infty}\limsup_{k\to\infty}\frac{1}{n_ik}\log Z_{a,e}^{k(n_i+l)}=P_\varphi(T_\psi).
    \end{align}

On the other hand, note that
\begin{align*}
    \phi_i(x_i)=\sum_{m\in\mathbb{Z}^k}e^{\langle m,x_i\rangle}\left[\frac{1}{P_{n_i}(1)}\sum_{\substack{v\in \alpha_{n_i}(a)\\\overline{\psi}_{n_i}(v)=m}}e^{\varphi_{n_i}(T^{-1}_v\xi)}\right]&\geq \frac{1}{P_{n_i}(1)}\sum_{\substack{v\in \alpha_{n_i}(a)\\\overline{\psi}_{n_i}(v)=0}}e^{\varphi_{n_i}(T^{-1}_v\xi)}\\
    &\geq C^{-2}\frac{1}{P_{n_i}(1)}\sum_{\substack{T^{n_i}(x)=x\\\overline{\psi}_{n_i}(x)=0\\x\in v}}e^{\varphi_{n_i}(x)},
\end{align*}
for some $v\in \alpha_{n_i}(a)$. 
By Corollary~\ref{cor:69}, $\lim_{i\to\infty}\frac{1}{n_i}\log \phi_i(x_i)$ 
exists and is finite. 
Thus, taking the logarithm and the limit in the previous displayed equation,
\begin{align*}
    \lim_{i\to\infty}\frac{1}{n_i}\log\phi_i(x_i)\geq -P_\varphi(T)+P_\varphi({T}_\psi^{\rm ab}).
\end{align*}
Using \eqref{eq:spectrumDS}, we conclude that
\begin{align}\label{eq:2-radius}
    \lim_{i\to\infty}\frac{1}{n_i}\log\lambda(G,m_{n_i})+P_\varphi(T)\geq P_\varphi(T_\psi^{\rm ab}), 
\end{align}
Combining \eqref{eq:1-radius}  and \eqref{eq:2-radius}, we obtain that
\begin{align}\label{eq:radius-final}
    P_\varphi(T_\psi^{\rm ab})\leq \lim_{i\to\infty}\frac{1}{n_i}\log\lambda(G,m_{n_i})+P_\varphi(T)\leq P_\varphi(T_\psi).
\end{align}
As $P_\varphi(T_\psi)\leq P_\varphi(T_\psi^{\rm ab})$ is always true, we obtain \eqref{eq: equal-abelian}.
The second part of the proposition follows directly by \eqref{eq: equal-abelian} and \eqref{eq:radius-final}.
\end{proof}

\end{document}